\documentclass{article}
\usepackage[utf8]{inputenc}
\usepackage{graphicx}
\usepackage{amsmath}
\usepackage{amsthm}
\usepackage{bbm}
\usepackage{dsfont}
\usepackage{caption}
\usepackage[norelsize,ruled,vlined,commentsnumbered]{algorithm2e}
\usepackage{times}
\usepackage{dsfont}
\usepackage{enumerate}
\usepackage{algpseudocode}
 \usepackage{xcolor}
\usepackage{amssymb}
\usepackage{amsbsy}
\usepackage{geometry}
\usepackage{array}
\def\P{\mathbb{P}}
\def\E{\mathbb{E}}

\renewcommand{\d}{\mathrm{d}}

\newcommand{\N}{\mathbf{N}_{M}}        		
\newcommand{\hatN}{\widehat{\mathbf{N}}_{M}} 	
\newcommand{\R}{\mathbb{R}}                   		
\newcommand{\C}{\mathbf{C}}

\newcommand{\mbf}[1]{\mathbf #1}
\newcommand{\eq}{\begin{equation}}
\newcommand{\qe}{\end{equation}}

\renewcommand{\l}{\mbf L_{M}}
\newcommand{\hatl}{\widehat{\mbf L}_{M}}

\newcommand{\B}{\mathbb B}

\newcommand{\X}{\mathcal X}	
\newcommand{\Y}{\mathcal Y}				
\renewcommand{\L}{\mathcal L^{D}}					
\newcommand{\Proj}{\mathfrak P}				

\newcommand{\Q}{\mathbf Q}				
\newcommand{\F}{\mathfrak F}	
\newcommand{\M}{\mathbf{M}_{M}} 			
\newcommand{\hatM}{\widehat{\mathbf{M}}_{M}}
\newcommand{\U}{\mathbf{U}}
\newcommand{\V}{\mathbf{V}}
\newcommand{\f}{\mathbf{f}}

\newcommand{\1}{\mathds{1}}

\newcommand{\Rbf}{\mathbf{R}}
     
\newcommand{\one}{\mathds{1}}
\renewcommand{\O}{\mathbf{O}_{M}}			
\newcommand{\hatO}{\widehat{\mathbf{O}}_{M}}	
\newcommand{\Ostar}{\mathbf{O}_{\star}}		
\renewcommand{\k}{{\kappa}}
\renewcommand{\B}{\mathbf{B}}

\newcommand{\D}[1]{\mathfrak{Diag}[{#1}]}        
\newcommand{\Pj}{\mathbf P_{M}}			
\newcommand{\hatPj}{\widehat{\mathbf P}_{M}}

\newcommand{\sv}{\mathrm{sv}}
\newcommand{\md}{\mathrm{md}}
\newcommand{\psg}{\mathbb{G}_{\mathrm{ps}}}
\newcommand{\gap}{\mathbb{G}}
\newcommand{\sigmaFamily}{\sigma_{K,\F^{\star}}}
\newcommand{\mixtime}{\mathbb{T}_{\mathrm{mix}}}
\newcommand{\concentration}{\mathcal{C}_{\star}(\Qstar,\delta')}
\newcommand{\borneN}{\mathbf{N}_{0}(\Q_{\star},\F^{\star},\Phi_{M},\delta,\delta')}
\newcommand{\borneNfinale}{\mathbf{N}_{1}(\Q_{\star},\F^{\star},\Phi_{M},\delta,\delta')}
\newcommand{\borneNfinalebis}{\mathbf{N}_{2}(\Q_{\star},\F^{\star},\Phi_{M},\delta,\delta')}
\newcommand{\borneNfinalebisbis}{\mathbf{N}_{3}(\Q_{\star},\F^{\star},\Phi_{M},\delta,\delta')}

\newcommand{\rmd}{\mathrm{d}}	
\newcommand{\Qstar}{\mathbf Q_{\star}}	
\newcommand{\Qhat}{\mathbf{\widehat{Q}}}
\newcommand{\Uhat}{\mathbf{\widehat{U}}}
\newcommand{\Bhat}{\mathbf{\widehat{B}}}
\newcommand{\Chat}{\mathbf{\widehat{C}}}					
\newcommand{\Fstar}{\mathrm{F}^{\star}}
\newcommand{\Fhat}{\widehat{\mathrm{F}}}

\newcommand{\fhat}{\widehat{f}}
\newcommand{\pistar}{\pi^{\star}}
\newcommand{\pistarmin}{\pistar_{\mathrm{min}}}
\newcommand{\pihat}{\widehat{\pi}}
\newcommand{\fstar}{f^{\star}}	
\newcommand{\frakFstar}{\mathfrak{F}^{\star}}

\newcommand{\dstar}{\delta^{\star}}	
\renewcommand{\L}{\mathcal L^{D}}	

\newcommand{\filtstar}[1]{\phi^{\star}_{#1}}
\newcommand{\smoothstar}[2]{\phi^{\star}_{#1|#2}}				
\newcommand{\filthat}[1]{\widehat{\phi}_{#1}}
\newcommand{\smoothhat}[2]{\widehat{\phi}_{#1|#2}}

\newcounter{hypH}
\newenvironment{hypH}{\refstepcounter{hypH}\begin{itemize}
\item[{[\bf H\arabic{hypH}]}]}{\end{itemize}}

\newcounter{hypHprime}
\newenvironment{hypHprime}{\refstepcounter{hypHprime}\begin{itemize}
\item[{[\bf H\arabic{hypHprime}']}]}{\end{itemize}}

\def\eqsp{\,}
\newcommand{\eqdef}{\ensuremath{:=}}

\newtheorem{remark}{Remark}[section]
\newtheorem{theorem}{Theorem}[section]
\newtheorem{proposition}[theorem]{Proposition}
\newtheorem{corollary}[theorem]{Corollary}
\newtheorem{lemma}[theorem]{Lemma}

\begin{document}

\author{Yohann {De Castro}\footnotemark[1]\and \'Elisabeth {Gassiat}\footnotemark[1]\and Sylvain {Le Corff}\footnotemark[1]}

\footnotetext[1]{Laboratoire de Math\'ematiques d'Orsay (CNRS UMR 8628), Universit\'e Paris-Sud 11, F-91405 Orsay Cedex, France.}

\title{Consistent estimation of the filtering and marginal smoothing distributions in nonparametric hidden Markov models}

\maketitle

\begin{abstract}
In this paper, we consider the filtering and smoothing recursions in nonparametric finite state space hidden Markov models (HMMs) when the parameters of the model are unknown and replaced by estimators. We provide an explicit and time uniform control of the  filtering and smoothing errors in total variation norm  as a function of the parameter estimation errors. We prove that the risk for the filtering and smoothing errors may be uniformly upper bounded by the risk of the estimators. It has been proved very recently that statistical inference for finite state space nonparametric HMMs is possible. We study how the recent spectral methods developed in the parametric setting may be extended to the nonparametric framework and we give explicit upper bounds for the $\mathrm{L}^{2}$-risk of the nonparametric spectral estimators. When the observation space is compact, this provides explicit rates for the filtering and smoothing errors in  total variation norm. The performance of the spectral method is assessed with  simulated data for both the estimation of the (nonparametric) conditional distribution of the observations  and the estimation of the marginal smoothing distributions.
\end{abstract}

\section{Introduction}
\label{sec:introduction}
\subsection{Context and motivations}
Hidden Markov models are popular time evolving models to depict practical situations in a variety of applications such as economics, genomics, signal processing and image analysis, ecology, environment, speech recognition, see \cite{douc:moulines:stoffer:2013} for a recent overview of HMMs. Finite state space HMMs are stochastic processes $(X_j,Y_j)_{j\geq 1}$ such that $(X_j)_{j\geq 1} $ is a Markov chain with finite state space $\X$ and $(Y_j)_{j\ge 1}$ are random variables with general state space $\Y$, independent conditionally on $(X_j)_{j\geq 1}$ and such that for all $\ell\ge 1$, the conditional distribution of $Y_\ell$ given $(X_j)_{j\ge 1}$ depends on $X_\ell$ only. The observations are $Y_{1:n} \eqdef (Y_1,\cdots,Y_n)$ and the associated states $X_{1:n} \eqdef (X_1,\cdots,X_n)$ are unobserved. The parameters of the model are the initial distribution $\pistar$ of the hidden chain, the transition matrix of the hidden chain $\Qstar$ and the conditional distribution of $Y_{1}$ given $X_{1}=x$ for all possible $x\in\X$ which are often called emission distributions. 

In many applications of finite state space HMMs (e.g. digital communication or speech recognition), it is of utmost importance to infer the sequence of hidden states.  Such inference usually involves the computation of the posterior distribution of a set of hidden states $X_{k:k'}$, $1\le k \le k'\le n$, given the observations $Y_{1:s}$, $1\le s\le n$.
When the initial distribution of the hidden chain, its transition matrix and the conditional distribution of the observations are known, this task can be efficiently done using the forward-backward algorithm described in \cite{baum:petrie:soules:weiss:1970} and \cite{rabiner:1989}. 
 In this paper, we focus on the estimation of the filtering distributions $\P(X_{k}=x|Y_{1:k})$ and marginal smoothing distributions $\P(X_{k}=x|Y_{1:n})$ for all $1<k\le n$ when the parameters of the HMM are unknown and replaced by estimators.
Indeed, it has been proved very recently that inference in finite state space nonparametric HMMs is possible, see \cite{gassiat2013finite}. 

\subsection{Contribution}

The aim of our paper is twofold. 
\begin{enumerate}[-]
\item First, we study how the parameter estimation error propagates to the error made on the estimation of filtering and smoothing distributions. Although replacing parameters by their estimators to compute posterior distributions and infer the hidden states is usual in applications, theoretical results to support this practice are very few regarding the accuracy of the estimated posterior distributions. We are only aware of \cite{evendar:kakade:mansour:2007} whose results are restricted to the filtering distribution in a parametric setting. When the parameters of the HMM are known, the forward-backward algorithm  can be extended to general state space HMMs or when the cardinality of $\X$ is too large using computational methods such as Sequential Monte Carlo methods (SMC), see \cite{delmoral:2004,doucet:defreitas:gordon:2001} for a review of these methods. In this context, the Forward Filtering Backward Smoothing \cite{kitagawa:1996,hurzeler:kunsch:1998,doucet:godsill:andrieu:2000} and  Forward Filtering Backward Simulation \cite{godsill:doucet:west:2004} algorithms  have been intensively studied, with the objective of quantifying the error made when the filtering and marginal smoothing distributions are replaced by their Monte Carlo approximations. 
These algorithms and some extensions have been analyzed theoretically recently, see for instance \cite{delmoral:doucet:singh:2010,douc:garivier:moulines:olsson:2011,dubarry:lecorff:2013,olsson:westerborn:2014}.  SMC methods may also be used in algorithms when the parameters of the HMM are unknown  to perform maximum likelihood parameter estimation, see \cite{kantas:doucet:singh:maciejowski:chopin:2014} for on-line and off-line Expectation Maximization and gradient ascent based algorithms.  Part of our analysis of the filtering and smoothing distributions is based on the same approach as in those papers and requires sharp forgetting properties of HMMs.
\item Second, we  extend spectral methods to a nonparametric setting and give explicit control of the $\mathrm{L}^2$-risk of the estimators. Such estimators may then be used in the computation of posterior distributions. In latent variable models such as HMMs, spectral methods are popular since they lead to algorithms that are not sensitive to a chosen initial estimate. Indeed, standard estimation methods for HMMs are based on the Expectation-Maximization (EM) algorithm, which faces intrinsic limitations hard to circumvent such as slow convergence and suboptimal local optima. Extending spectral methods to nonparametric HMMs is thus very useful. In particular, they may be  used to provide a preliminary estimator as starting point in a EM algorithm. They are also used in a refinement procedure proposed in \cite{decastro:gassiat:lacour:2015}. To the best of our knowledge, the spectral method has not been extended nor studied yet in the nonparametric framework.  \\
We start from the works of Anandkumar, Hsu, Kakade and Zhang on spectral methods in the parametric frame. Their papers \cite{hsu2012spectral,anandkumar2012method} present an efficient algorithm for learning parametric HMMs or more generally finitely many linear functionals of the parameters of the HMM. Thus, it is possible to use spectral methods to estimate the projections  of the emission distributions onto nested subspaces of increasing complexity.
Our work brings a new quantitative insight on the tradeoff between sampling size and approximation complexity for spectral estimators.  We provide a nonasymptotic precise upper bound of the risk for the variance term with respect to the number of observations and the complexity of the approximating subspace.  
\end{enumerate}

\subsection{Outline of the paper}
In section \ref{sec:consistency},
we provide an explicit control of the total variation filtering and smoothing errors as a function of the parameter estimation error, see Propositions \ref{prop:filtering} and \ref{prop:smoothing}.
We detail the application of these preliminary results in the parametric context, see Theorem \ref{theo:para}, and  in the nonparametric context, see Theorem \ref{theo:nonpara} where we prove that the uniform rate of convergence for the filtering and smoothing errors is driven by the $\mathrm{L}^{1}$-risk of the nonparametric estimator of the emission distributions.
In Section \ref{sec:spectral:hmm}, we explain how spectral methods can be extended to the nonparametric frame and we provide the nonasymptotic control of the variance term in Theorem \ref{thm:spectral}. This leads to the asymptotic behavior proved in Corollary \ref{cor:spectral}, which may be invoked when spectral methods are used in the computation of posterior distributions, see Corollary \ref{th:uniform:consistency} stated in the previous section. 
Finally, in Section \ref{sec:exp} we show  the performance of the spectral method with  simulated data for both the estimation of the (nonparametric) conditional distribution of the observations  and the estimation of the marginal smoothing distributions. All detailed proofs are given in the appendices.

\section{Main results}
\label{sec:consistency}
\subsection{Notations and setting}
In the sequel, it is assumed that the cardinality $K$ of $\X$ is known (for ease of notation, $\X$ is set to be $\{1,\ldots,K\}$) and that $\Y$ is a subset of $ \R^D$ for a positive integer $D$. Denote by $\mathcal P(\X)$ the space of probability measures on $\X$ and write $\mathcal L^{D}$ the Lebesgue measure on $\Y$. For all $n\ge 1$ and all $x\in\X$, the density of the conditional distribution of $Y_n$ given $X_n=x$ with respect to $\mathcal L^{D}$ is written $f_x^{\star}$. Consider the following assumptions on the hidden chain.
\begin{hypH}
\label{assum:Qstar}
\begin{enumerate}[a)]
\item \label{assum:Qstar:rank}The transition matrix $\Qstar$ has full rank.
\item \label{assum:Qstar:deltastar}$\displaystyle\dstar \eqdef \min_{1\le i,j\le K} \Qstar(i,j) >0$.
\end{enumerate}
\end{hypH}

\begin{hypH}
\label{assum:stationary}
The initial distribution $\pistar \eqdef (\pistar_{1},\ldots,\pistar_{K})$ is the stationary distribution.
\end{hypH}
\begin{remark}
Note that under {\bf [H\ref{assum:Qstar}]-\ref{assum:Qstar:deltastar})} and {\bf [H\ref{assum:stationary}]}, for all $k\in\X$, $\pistar_{k}\geq\delta^{\star}>0$.
\end{remark}
\begin{remark}
Assumptions {\bf [H\ref{assum:Qstar}]-\ref{assum:Qstar:rank})} and {\bf [H\ref{assum:stationary}]} appear in spectral methods, see for instance \cite{anandkumar2012method,hsu2012spectral}, and in identifiability issues, see for instance \cite{alexandrovich2014nonparametric,allman2009identifiability,gassiat2013finite}.
\end{remark}
For all $y\in\Y$, define $c_{\star}(y)$ by
\begin{equation}
\label{eq:def:cstar}
c_{\star}(y) \eqdef \min_{x\in\X} \sum_{x'\in\X}\Qstar(x,x')\fstar_{x'}(y)\eqsp.
\end{equation}
For all $y_{1:n}\in\Y^n$, the filtering distributions $\filtstar{k}(\cdot,y_{1:k})$ and marginal smoothing distributions $\smoothstar{k}{n}(\cdot,y_{1:n})$ may be computed explicitly for all $1\le k\le n$ using the forward-backward algorithm of \cite{baum:petrie:soules:weiss:1970}. In the forward pass, the filtering distributions $\filtstar{k}$ are updated recursively using, for all $x\in\X$,
\begin{equation}
\label{eq:recursion-fil}
\filtstar{1}(x,y_1) \eqdef  \frac{\pistar(x)\fstar_x(y_1)}{\sum_{x'\in\X}\pistar(x')\fstar_{x'}(y_1)} \;\;\mbox{and}\;\;\filtstar{k}(x,y_{1:k}) \eqdef \frac{\sum_{x'\in\X}\Qstar(x',x)\fstar_x(y_k)\filtstar{k-1}(x',y_{1:k-1})}{\sum_{x',x''\in\X}\Qstar(x',x'')\fstar_{x''}(y_k)\filtstar{k-1}(x',y_{1:k-1})}\eqsp.
\end{equation}
Note that for all $1\le k \le n$, $\filtstar{k}(x,Y_{1:k}) = \P(X_{k}=x|Y_{1:k})$. In the backward pass, the marginal smoothing distributions may be updated recursively using, for all $x\in\X$,
\begin{equation}
\label{eq:recursion-smooth}
\smoothstar{n}{n}(x,y_{1:n}) \eqdef  \filtstar{n}(x,y_{1:n}) \quad\mbox{and}\quad \smoothstar{k}{n}(x,y_{1:n}) \eqdef \sum_{x'\in \X} B^{\star}_{\filtstar{k}(\cdot,y_{1:k})}(x',x)\smoothstar{k+1}{n}(x',y_{1:n})\eqsp,
\end{equation}
where, for all $u,v\in\X$ and all $1\le k\le n$, 
\[
B^{\star}_{\filtstar{k}(\cdot,y_{1:k})}(u,v) \eqdef \frac{\Qstar(v,u)\filtstar{k}(v,y_{1:k})}{\sum_{z\in \X}\Qstar(z,u)\filtstar{k}(z,y_{1:k})}\eqsp.
\]
Note that for all $1\le k \le n$, $\filtstar{k|n}(x,Y_{1:n}) = \P(X_{k}=x|Y_{1:n})$.

\subsection{Preliminary results}
In this paper, the parameters $\pistar$, $\Qstar$ and $\fstar$ are unknown. Then, the recursive equations \eqref{eq:recursion-fil} and \eqref{eq:recursion-smooth} may  be applied replacing $\pistar$, $\Qstar$ and $\fstar$ by some estimators $\pihat$, $\Qhat$ and $\fhat$ to obtain approximations of the filtering and smoothing distributions. Using forgetting properties of the hidden chain, we are able  to obtain an  upper bound of the filtering errors and of  the marginal smoothing errors by terms involving only the estimation errors of $\pistar$, $\Qstar$ and $\fstar$. These upper bounds are given in propositions~\ref{prop:filtering} and~\ref{prop:smoothing}. Their proofs are postponed to Appendix~\ref{sec:filtResults} and~\ref{sec:smoothResults}. Note that the upper bounds are given for any possible values $y_{1:k}$, $k\geq 1$ and may be applied to the set of observations for which filtering and smoothing distributions are estimated, whatever the set of observations used to estimate $\pistar$, $\Qstar$ and $\fstar$. Let $\|\cdot\|_{\mathrm{tv}}$ be the total variation norm,  $\|\cdot\|_{2}$ the euclidian norm and  $\|\cdot\|_{F}$ the Frobenius norm. For all $1\le k \le n$, denote by $\filthat{k}$ and $\filthat{k|n}$ the approximations of $\filtstar{k}$ and $\filtstar{k|n}$ obtained by replacing $\pistar$, $\Qstar$ and $\fstar$ by the estimators $\pihat$, $\Qhat$ and $\fhat$ in \eqref{eq:recursion-fil} and \eqref{eq:recursion-smooth}.
\begin{proposition}
\label{prop:filtering}
Assume {\bf [H\ref{assum:Qstar}]-\ref{assum:Qstar:deltastar})} and {\bf [H\ref{assum:stationary}]} hold. Then, for all $k\geq 1$ and all $y_{1:k}\in \Y^k$,
\begin{multline*}
\|\filtstar{k}(\cdot,y_{1:k}) -\filthat{k}(\cdot,y_{1:k})\|_{\mathrm{tv}}\le C_{\star}\left(\rho_{\star}^{k-1}\left\|\pistar-\pihat\right\|_{2}/\delta^{\star} + \|\Qstar-\Qhat\|_{F}/(\delta^{\star}(1-\rho_{\star}))\right. \\
+ \left.\sum_{\ell=1}^{k}\rho_{\star}^{k-\ell}c_{\star}^{-1}(y_{\ell})\max_{x\in\X}\left|f^{\star}_{x}(y_{\ell})-\hat{f}_{x}(y_{\ell})\right|\right)\eqsp,
\end{multline*}
where 
$\rho_{\star} \eqdef 1-\dstar/(1-\dstar)\eqsp$ and $C_{\star} \eqdef 4(1-\dstar)/\dstar$.
\end{proposition}

\begin{proposition}
\label{prop:smoothing}
Assume {\bf [H\ref{assum:Qstar}]-\ref{assum:Qstar:deltastar})} and  {\bf [H\ref{assum:stationary}]} hold. 
Then, for all $1\le k \le n$ and all $y_{1:n}\in\Y^n$,
\begin{multline*}
\|\filtstar{k|n}(\cdot,y_{1:n}) -\filthat{k|n}(\cdot,y_{1:n})\|_{\mathrm{tv}}\le C_{\star}\left(\rho_{\star}^{k-1}\|\pistar-\pihat\|_{2}/\delta^{\star} + [1/(1-\rho_{\star})+ 1/(1-\widehat{\rho})]\|\Qstar-\Qhat\|_{F}/\delta^{\star}\right.\\
\left.+ \sum_{\ell=1}^n(\widehat{\rho}\vee \rho_{\star})^{\,|\ell-k|}c_{\star}^{-1}(y_{\ell})\max_{x\in\X}\left|f^{\star}_{x}(y_{\ell})-\hat{f}_{x}(y_{\ell})\right|\right)\eqsp,
\end{multline*}
where 
$\widehat{\delta}\eqdef  \min_{x,x'} \Qhat(x,x') $ and $\widehat{\rho} \eqdef 1-\widehat{\delta}/(1-\widehat{\delta})$.
\end{proposition}

\subsection{Uniform consistency of  the posterior distributions}

Propositions~\ref{prop:filtering} and~\ref{prop:smoothing} are preliminary results that can be used to understand how the estimation errors made on the parameters of the HMM propagate upon the filtering and smoothing distributions.
We  assume that we are given a set of $p+n$ observations from the hidden Markov model driven by $\pistar$, $\Qstar$ and $\fstar$. The first $p$ observations are used to produce the estimators $\pihat$, $\Qhat$ and $\fhat$ while filtering and smoothing are performed with the last $n$ observations.
 In other words   the estimators $\pihat$, $\Qhat$ and $\fhat$ are measurable functions of $Y_{1:p}$ and  the objective is to estimate $\filtstar{k}(\cdot,Y_{p+1:p+k})$ and $\filtstar{k|n}(\cdot,Y_{p+k:p+n})$. 

\subsubsection{Parametric models}
In the parametric case, the hidden Markov model depends on a parameter $\theta_{\star}$ which lies in a subset of $\mathbb{R}^{q}$ for a given $q\ge 1$. In this situation, $\theta_{\star}$ may be estimated by $\widehat{\theta}\in \mathbb{R}^{q}$ and we may write $\pihat \eqdef \pi^{\widehat{\theta}}$, $\Qhat \eqdef Q_{\widehat{\theta}}$ and $\fhat \eqdef f^{\widehat{\theta}}$.  
 
\begin{theorem}
\label{theo:para}
Assume {\bf [H\ref{assum:Qstar}]} and  {\bf [H\ref{assum:stationary}]} hold. Assume also that 
 for all $x,x'\in\X$, $\theta\mapsto Q_{\theta}(x,x')$ is continuously differentiable with a bounded derivative in the neighborhood of $\theta_{\star}$ and that for all $x\in\X$ and all $y\in\Y$,  $\theta\mapsto f^{\theta}_x(y)$ is continuously differentiable in the neighborhood of $\theta_{\star}$ and such that the norm of its gradient is upper bounded in this neighborhood by a function $h_{x}$ such that $\int h_{x}(y)d\mathcal L^{D}(y)<+\infty$. Let  $\widehat{\theta}$ be a consistent estimator of $\theta_{\star}$. Then for any $1\le k \le n$, 
$$
\|\filtstar{k}(\cdot,Y_{p+1:p+k}) -\filthat{k}(\cdot,Y_{p+1:p+k})\|_{\mathrm{tv}} =  O_{\mathbb{P}}(\|\widehat{\theta}-\theta_{\star}\|_{2}) 
$$
and
$$
\|\filtstar{k|n}(\cdot,Y_{p+1:p+k}) -\filthat{k|n}(\cdot,Y_{p+1:p+k})\|_{\mathrm{tv}} =  O_{\mathbb{P}}(\|\widehat{\theta}-\theta_{\star}\|_{2})\eqsp. 
$$
\end{theorem}
The smoothness assumption in Theorem~\ref{theo:para} is usual to study the asymptotic distribution of the maximum likelihood estimator in parametric HMMs. By Theorem~\ref{theo:para},  tight bounds on the uniform convergence rate of $\|\filtstar{k}(\cdot,Y_{p+1:p+k}) -\filthat{k}(\cdot,Y_{p+1:p+k})\|_{\mathrm{tv}}$ and of
$\|\filtstar{k|n}(\cdot,Y_{p+1:p+k}) -\filthat{k|n}(\cdot,Y_{p+1:p+k})\|_{\mathrm{tv}}$
may be derived by controlling the estimation error $\|\widehat{\theta}-\theta_{\star}\|$.  There exist several results on this error term depending on the algorithm used to obtain $\widehat{\theta}$. For instance, \cite{tadic:2010} provides explicit upper bounds for this error term in the case  where $\widehat{\theta}$ is a recursive maximum likelihood estimator of $\theta_{\star}$, under additional assumptions on the model. 

\begin{proof}
First,  under {\bf [H\ref{assum:Qstar}]} and  {\bf [H\ref{assum:stationary}]}, the assumption on $\theta\mapsto Q_{\theta}(x,x')$ implies that  $\theta\mapsto \pi^{\theta}_x$ is continuously differentiable with a bounded derivative in the neihgborhood of $\theta_{\star}$. Notice also that $\sup_{k\geq 1} \rho_{\star}^{k-1}\leq 1$ and $\sup_{k\geq 1}\widehat{\rho}^{\,k-1}\leq 1$. 
Then using  Taylor expansion  we easily get that the first two terms of the upper bound in Propositions~\ref{prop:filtering} and~\ref{prop:smoothing} are $O_{\mathbb{P}}(\|\widehat{\theta}-\theta_{\star}\|_{2})$. 
 There just remains to control the last term for each of the upper bound in Propositions~\ref{prop:filtering} and~\ref{prop:smoothing}. Using a Taylor expansion, Cauchy-Schwarz inequality, and Proposition~\ref{prop:filtering}, we get that for any $1\le k \le n$,
\[
\|\filtstar{k}(\cdot,Y_{p+1:p+k}) -\filthat{k}(\cdot,Y_{p+1:p+k})\|_{\mathrm{tv}}\le O_{\mathbb{P}}(\|\widehat{\theta}-\theta_{\star}\|_{2}) + 
\|\widehat{\theta}-\theta_{\star}\|_{2}\sum_{\ell=1}^{k}\rho_{\star}^{k-\ell}c_{\star}^{-1}(Y_{p+\ell})\sum_{x\in\X}h_{x}(Y_{p+\ell})\eqsp.
\]
As the $(Y_{j})_{j\geq 1}$ are stationary with distribution having  density $\sum_{x\in\X}\pi^{\star}_{x}f^{\star}_{x} (y)\leq c_{\star}(y)/\delta^{\star}$, the random variable $\sum_{\ell=1}^{k}\rho_{\star}^{k-\ell}c_{\star}^{-1}(Y_{p+\ell})\sum_{x\in\X}h_{x}(Y_{p+\ell})$ is nonnegative and has expectation
upper bounded by
$$
\frac{1}{\delta^{\star}}\sum_{\ell=1}^{k}\rho_{\star}^{k-\ell}\sum_{x\in\X}\int h_{x}(y)d\mathcal L^{D}(y)\leq 
\frac{1-\delta^{\star}}{(\delta^{\star})^{2}}\sum_{x\in\X}\int h_{x}(y)d\mathcal L^{D}(y)
<+\infty\eqsp.
$$
Thus $\sum_{\ell=1}^{k}\rho_{\star}^{k-\ell}c_{\star}^{-1}(Y_{p+\ell})\sum_{x\in\X}h_{x}(Y_{p+\ell})=O_{\mathbb{P}}(1)$ so that we get the first point of Theorem \ref{theo:para}.
The result for the smoothing distributions follows the same lines after noticing that, for some $\epsilon >0$ such that $\rho_{\star}+\epsilon <1$, the event $\left\{\widehat{\rho}\geq \rho_{\star}+\epsilon\right\}$ has probability tending to $0$ as $p$ tends to infinity when $\widehat{\theta}$ is a consistent estimator of $\theta_{\star}$. 
\end{proof}

\subsubsection{Nonparametric models}
\label{sec:consistency:nonparametric}
We first state a general theorem that gives a control of the uniform consistency of the posterior distributions in terms of the risk of the nonparametric estimators. The theorem also holds in the parametric context, however, the parametric literature usually studies the distributional properties of the estimators,  while the nonparametric one studies mostly the risk. As usual in the hidden Markov model literature, the model parameters are identifiable up to permutations of the hidden states labels. Therefore, without loss of generality, the following results are stated indicating the prospective permutation of the states. Let $\mathcal S_{K}$ be the set of permutations of $\{1,\ldots,K\}$. If $\tau$ is a permutation, let $\mathbb P_{\tau}$ be the permutation matrix associated with $\tau$.

\begin{theorem}
\label{theo:nonpara}
Assume {\bf [H\ref{assum:Qstar}]-\ref{assum:Qstar:deltastar})} and  {\bf [H\ref{assum:stationary}]} hold. Then for all $n\geq 1$, for any permutation $\tau_{p}\in\mathcal S_{K}$, 
\begin{multline*}
\sup_{ 1\le k \leq n}\E\left[\|\filtstar{k}(\cdot,Y_{p+1:p+k}) -\filthat{k}^{\,\tau_p}(\cdot,Y_{p+1:p+k})\|_{\mathrm{tv}}\right]\\
 \le \frac{C_{\star}}{(\delta^{\star})^{2}}\left\{\E[\left\|\pistar-\mathbb P_{\tau_{p}}\pihat_{p}\right\|_{2}] + \E[\| \Q^{\star}-\mathbb P_{\tau_{p}}\Qhat_{p}\mathbb P_{\tau_{p}}^{\top}\|_{F}]
+\sum_{x\in\X}\E[\lVert f^{\star}_{x}-\fhat_{\tau_{p}(x)}\lVert_{1}]\right\}
\end{multline*}
and
\begin{multline*}
\sup_{ 1\le k \leq n}\E\left[\|\filtstar{k|n}(\cdot,Y_{p+1:p+n}) -\filthat{k|n}^{\,\tau_p}(\cdot,Y_{p+1:p+n})\|_{\mathrm{tv}}\right] \\
\leq  \frac{C_{\star}}{(\delta^{\star})^2}\left\{\E[\left\|\pistar-\mathbb P_{\tau_{p}}\pihat_{p}\right\|_{2}] + \E\big[\lVert \Q^{\star}-\mathbb P_{\tau_{p}}\Qhat_{p}\mathbb P_{\tau_{p}}^{\top}\lVert_{F}/\widehat\delta\big]
 +  \sum_{x\in\X}\E\left[\lVert f^{\star}_{x}-\fhat_{\tau_{p}(x)}\lVert_{1}/\widehat\delta\right] \right\}.
\end{multline*}
Here,  
$\filthat{k}^{\,\tau_p}$ and $\filthat{k|n}^{\,\tau_p}$ are the estimation of $\filtstar{k}$ and $\filtstar{k|n}$ based on $\mathbb P_{\tau_{p}}\Qhat\,\mathbb P_{\tau_{p}}^{\top}$, $\mathbb P_{\tau_{p}}\pihat$ and $\fhat_{\tau_{p}(x)}$, for all $x\in\X$.
\end{theorem} 

\begin{proof}
For any $x\in \X$ and any $1\le \ell\le n$,
\begin{align*}
\mathbb{E}\left[c_{\star}^{-1}(Y_{p+\ell})\left|f^{\star}_{x}(Y_{p+\ell})-\hat{f}_{\tau_{p}(x)}(Y_{p+\ell})\right|\right] &= \mathbb{E}\left[\mathbb{E}\left[c_{\star}^{-1}(Y_{p+\ell})\left|f^{\star}_{x}(Y_{p+\ell})-\hat{f}_{\tau_{p}(x)}(Y_{p+\ell})\right|\middle|Y_{1:p+\ell-1}\right]\right]\eqsp,
\end{align*}
with
\[
\mathbb{E}\left[c_{\star}^{-1}(Y_{p+\ell})\left|f^{\star}_{x}(Y_{p+\ell})-\hat{f}_{\tau_{p}(x)}(Y_{p+\ell})\right|\middle|Y_{1:p+\ell-1}\right] = \int \left|f^{\star}_{x}(z)-\hat{f}_{\tau_{p}(x)}(z)\right|c_{\star}^{-1}(z)g_{\ell}(z) \d z\eqsp,
\]
where $g_{\ell}(z) \eqdef \sum_{x_{\ell-1},x_{\ell}\in \X}\filtstar{\ell-1}(x_{\ell-1},Y_{p+1:p+\ell-1})\Qstar(x_{\ell-1},x_{\ell})\fstar_{x_{\ell}}(z)$. By {\bf [H\ref{assum:Qstar}]-\ref{assum:Qstar:deltastar})} and \eqref{eq:def:cstar}, $c_{\star}^{-1}(z)g_{\ell}(z) \le (1-\delta^{\star})/\delta^{\star}$ and 
\[
\mathbb{E}\left[c_{\star}^{-1}(Y_{p+\ell})\left|f^{\star}_{x}(Y_{p+\ell})-\hat{f}_{x}(Y_{p+\ell})\right|\middle|Y_{1:p+\ell-1}\right]  \leq (1-\delta^{\star})\|f^{\star}_{x}-\hat{f}_{\tau_{p}(x)}\|_1/\delta^{\star}\eqsp.
\]
Therefore,  the result  for the filtering distributions comes from taking the supremum and then the expectation in the upper bound of Proposition~\ref{prop:filtering}.
The proof for the smoothing distributions follows the same steps.
\end{proof}

What comes out  in Theorem~\ref{theo:nonpara} is a control driven by the $\mathrm{L}^{1}$-risk of the emission densities. In Section~\ref{sec:spectral:hmm}, we propose a spectral method to obtain, in the nonparametric context, estimators of the transition matrix, the stationary distribution and the emission densities. The general idea is that of projection methods, so that at the end we obtain a control on the  $\mathrm{L}^{2}$-risk of the emission densities. This control can be easily transfered whenever $\Y$ is a compact subset of $\R^{D}$, since in such a case, for some $C(\Y)>0$ we have, for any square integrable functions $h_{1}$ and $h_{2}$,
\begin{equation}
\label{eq:L1L2}
\lVert h_{1}-h_{2}\lVert_{1} \leq C(\Y)\lVert h_{1}-h_{2}\lVert_{2}.
\end{equation}
We end this section by setting the result that follows when using the spectral estimators. Let $(M_{r})_{r\geq1}$ be an increasing sequence of integers, and let $(\Proj_{M_{r}})_{r\geq 1}$ be a sequence of nested subspaces such that their union is dense in $\mathrm{L}^{2}(\Y,\L)$.  Let $\Phi_{M_{r}}:=\{\varphi_{1},\ldots,\varphi_{M_{r}}\}$ be an orthonormal basis of $\Proj_{M_{r}}$. Note that for all $f\in\mathrm{L}^{2}(\Y,\L)$,
\begin{equation}
\label{eq:ConvergenceL2}
\lim_{p\to\infty}\sum_{m=1}^{M_{r}}\langle f,\varphi_{m}\rangle\varphi_{m}= f\,,
\end{equation}
in $\mathrm{L}^{2}(\Y,\L)$.  Note also that changing $M_r$ may change all functions $\varphi_{r}$, $1\leq m \leq M_r$ in the basis $\Phi_{M_r}$, which will not be indicated in the notation for better clarity.
We shall also drop the index $r$ and write $M$ instead of $M_{r}$.\\
The spectral estimators of the emission densities  will be projection estimators. Let us denote $f^{\star}_{M,1},\ldots, f^{\star}_{M,K}$ the projections of the emission densities on the space $\Proj_{M}$, that is, for $x\in\X$,
$$
f^{\star}_{M,x}=\sum_{m=1}^{M}\langle f^{\star}_{x},\varphi_{m}\rangle\varphi_{m}.
$$
We need a further assumption, which, together with {\bf [H\ref{assum:Qstar}]-\ref{assum:Qstar:deltastar})} and  {\bf [H\ref{assum:stationary}]},  has been proved sufficient to get identifiability in nonparametric HMMs, see \cite{gassiat2013finite}.
\begin{hypH}
\label{assum:linearly:ind}
The family of emission densities $\F^{\star}:=\{f^{\star}_{1},\ldots,f^{\star}_{K}\}$ is linearly independent.
\end{hypH}
Finally, the following quantity is needed in the control of the $\mathrm{L}^{2}$-risk of the spectral estimators. For any $M$, define 
\begin{equation}
\label{eq:def:eta3}
\eta_{3}^{2}(\Phi_{M})\eqdef\sup_{y,y'\in\Y^{3}}\sum_{a,b,c=1}^{M}\left(\varphi_{a}(y_{1})\varphi_{b}(y_{2})\varphi_{c}(y_{3})-\varphi_{a}(y'_{1})\varphi_{b}(y'_{2})\varphi_{c}(y'_{3})\right)^{2}\,.
\end{equation}
Applying Theorem \ref{theo:nonpara} and \eqref{eq:L1L2} we get the following corollary whose proof is omitted: the first point is an application of Corollary \ref{cor:spectral}, the second point is obtained following the same lines as the proof of Corollary \ref{cor:spectral}.

\begin{corollary}
\label{th:uniform:consistency}
Assume {\bf [H\ref{assum:Qstar}]-[H\ref{assum:linearly:ind}]} hold. Assume also that for all $x\in\X$, $\fstar_x\in\mathrm{L}^{2}(\Y,\L)$. Let $M_{p}$ be a sequence of integers tending to infinity such that $\eta_3(\Phi_{M_p}) = o(\sqrt{p/\log p})$. For each $p$, define $\fhat$,
$\Qhat$ and $\pihat$ as the estimators obtained by the spectral algorithm given in Section~\ref{sec:spectral:hmm} with this choice of $M_p$. Then, there exists a sequence of permutations $\tau_{p}\in\mathcal S_{K}$ such that
$$
\E\left[\sup_{ k \geq 1}\|\filtstar{k}(\cdot,Y_{p+1:p+k}) -\filthat{k}^{\,\tau_p}(\cdot,Y_{p+1:p+k})\|_{\mathrm{tv}}\right]=O\big(\eta_3(\Phi_{M_p})\sqrt {\log p/p}
+ \sum_{x\in\X}\|f^{\star}_{x}-f^{\star}_{M_{p},x}\|_{2}\big)
$$
and
$$
\E\left[\sup_{1\leq k \leq n}\|\filtstar{k|n}(\cdot,Y_{p+1:p+n}) -\filthat{k|n}^{\,\tau_p}(\cdot,Y_{p+1:p+n})\|_{\mathrm{tv}}\right] 
=O\big(\eta_3(\Phi_{M_p})\sqrt {\log p/p}
+ \sum_{x\in\X}\|f^{\star}_{x}-f^{\star}_{M_{p},x}\|_{2}\big)
.
$$
\end{corollary}

One may consider the following standard examples.
\begin{enumerate}[-]
\item {\bf (Spline)} The space of piecewise polynomials of degree bounded by $d_r$ based on the regular partition with $p_{r}^{D}$ regular pieces on $\Y$. 
It holds that $M_{r}=(d_{r}+1)^{D}p_{r}^{D}$. 
\item {\bf (Trig.)} The space of real trigonometric polynomials on $\Y$ with degree less than $r$. It holds that $M_{r}=(2r+1)^{D}$.
\item {\bf (Wav.)} A wavelet basis $\Phi_{M_{r}}$ of scale $r$ on $\Y$,  see
\cite{MR1228209}. 
It holds that $M_{r}=2^{(r+1)D}$.
\end{enumerate}
In those examples, there exists a constant $C_{\eta}>0$ such that $\eta_{3}(M)\leq C_{\eta}M^{k/2}$, so that the rate of uniform convergence for the posterior probabilities is
$
O\big({M_p}^{3/2}\sqrt {\log p/p}
+ \sum_{x\in\X}\|f^{\star}_{x}-f^{\star}_{M_{p},x}\|_{2}\big)
$.

\section{Nonparametric spectral estimation of HMMs}
\label{sec:spectral:hmm}

\subsection{Description of the spectral method}
\label{sec:SpectralMethod}
This section describes a tractable approach to get nonparametric estimators of the emission densities and of the transition matrix. Our procedure relies on the estimation of the
projections of the emission laws onto nested subspaces of increasing complexity. This allows to illustrate the uniform consistency result provided in the previous section.  

Recall that $(\Proj_{M_{r}})_{r\geq 1}$ is a sequence of nested subspaces of $\mathrm{L}^{2}(\Y,\L)$ associated with their orthonormal basis $(\Phi_{M_{r}})_{r\geq 1}$. Since projections are linear functionals of the distributions, it is possible to use spectral methods to estimate the projections of the emission distributions on the basis $\Phi_{M}$ for each M. To this end, our approach is based on the work described in \cite{anandkumar2012method}. In particular, we follow their strategy to get an estimation of the emission densities. However, the dependency in the dimension is of crucial importance in the nonparametric framework and it has not been addressed in \cite{anandkumar2012method}. Hence, we present in Theorem C.3 a new quantitative version of the work \cite{anandkumar2012method} that accounts for the dimension M. Moreover, the authors of \cite{anandkumar2012method} invoke a way of estimating the transition matrix $\Qstar$ but they do not give any theoretical garantees regarding this estimator. In this paper, we introduce a slightly different estimator that is based on a surrogate $\tilde\pi$ (see Step 8 of Algorithm 1) of the stationary distribution. Our estimator (see Step 9 of Algorithm 1) is then build from the "observable" operator (rather than its left singular vectors as done in \cite{anandkumar2012method}). Eventually, Theorem C.2 gives the theoretical guarantees of our estimator of the transition matrix and its stationnary distribution.

The computation of those estimators is particularly simple: it is based on one singular value decomposition, matrix inversions and one diagonalization. It is proved in Theoremn~\ref{thm:ThmMain1} and~\ref{thm:HSK} that, with overwhelming probability, all the matrix inversions and the diagonalization can be done rightfully.

For all $(p\times q)$ matrix $A$ with $p\geq q$, denote by $\sigma_{1}(A)\geq\sigma_{2}(A)\geq\ldots\geq\sigma_{q}(A)\geq0$ its singular values and $\lVert\cdot\lVert$ its operator norm. When $A$ is invertible, let $\kappa(A):=\sigma_{1}(A)/\sigma_{q}(A)$ be its condition number. $A^{\top}$ is the transpose matrix of $A$, $A(\ell,\ell')$ its $(\ell,\ell')$th entry, $A(\ldotp,\ell)$ its $\ell$th column and $A(k,\ldotp)$ its $k$th line. When $A$ is a $(p\times p)$ diagonalizable matrix, its eigenvalues are written $\lambda_{1}(A)\geq\lambda_{2}(A)\geq\ldots\geq\lambda_{p}(A)$. For any $1\leq q\leq+\infty$, $\lVert\cdot\lVert_{q}$ is the usual $\mathrm{L}^q$ norm for vectors. For any row or column vector $v$, denote by $\D v$ the diagonal matrix with diagonal entries $v_{i}$. The following vectors, matrices and tensors are used throughout the paper:
\begin{enumerate}[-]
\item ${\l}\in\R^{M}$ is the projection of the distribution of one observation on the basis $\Phi_{M}$: for all $a\in\{1,\ldots,M\}$, $\l(a)\eqdef\E[\varphi_{a}(Y_{1})]$ ;
\item $\N\in\R^{M\times M}$ is the joint distribution of two consecutive observations: for all $(a,b)\in\{1,\ldots,M\}^{2}$, $\N(a,b)\eqdef\E[\varphi_{a}(Y_{1})\varphi_{b}(Y_{2})]$ ;
\item $\M\in\R^{M\times M\times M}$ is the joint distribution of three consecutive observations: for all $(a,b,c)\in\{1,\ldots,M\}^{3}$, $\M(a,b,c)\eqdef\E[\varphi_{a}(Y_{1})\varphi_{b}(Y_{2})\varphi_{c}(Y_{3})]$ ;
\item $\O\in\R^{M\times K}$ is the conditional distribution of one observation on the basis $\Phi_{M}$: for all $(m,x)\in\{1,\ldots,M\}\times\X$, $\O(m,x)\eqdef\E[\varphi_{m}(Y_{1})|X_{1}=x]=\langle f_{x}^{\star},\varphi_{m}\rangle$ ;
\item For all $x\in\X$, $\fstar_{M,x}$ is the projection of the emission laws on the subspace $\Proj_{M}$: , $\fstar_{M,x}\eqdef\sum_{m=1}^{M}\O(m,x)\varphi_{m}$. Write $\f^{\star}_{M}\eqdef( f^{\star}_{M,1},\ldots, f^{\star}_{M,K})$ ;
\item $\Pj\in\R^{M\times M}$ is the joint distribution of  $(Y_{1},Y_{3})$: for all $(a,c)\in\{1,\ldots,M\}^{2}$, $\Pj(a,c)\eqdef\E[\varphi_{a}(Y_{1})\varphi_{c}(Y_{3})]$.
\end{enumerate}


\begin{algorithm}[t]
  \SetAlgoLined
  \KwData{An observed chain $(Y_{1},\ldots,Y_{p+2})$ and a number of hidden states $K$.}
  \KwResult{Spectral estimators $\pihat$, $\Qhat$ and $(\fhat_{M,x})_{x\in\X}$.}
    \BlankLine
\begin{enumerate}[{\bf [Step 1]}]
\item For all $a,b,c$ in $\{1,\ldots,M\}$, consider the following empirical estimators: $\hatl(a)\eqdef\sum_{s=1}^{p}\varphi_{a}(Y_{s})/p$, $\hatM(a,b,c)\eqdef \sum_{s=1}^{p}\varphi_{a}(Y_{s})\varphi_{b}(Y_{s+1})\varphi_{c}(Y_{s+2})/p$, $\hatN(a,b)\eqdef \sum_{s=1}^{p}\varphi_{a}(Y_{s})\varphi_{b}( Y_{s+1})/p$ and $\hatPj(a,c)\eqdef \sum_{s=1}^{p}\varphi_{a}(Y_{s})\varphi_{c}(Y_{s+2})/p$.
\item
Let $\Uhat$ be the $M\times K$ matrix of orthonormal right singular vectors of $\hatPj$ corresponding to its top $K$ singular values. 
\item 
For all $b\in\{1,\ldots,M\}$, set $\Bhat(b)  \eqdef (\Uhat^{\top}\hatPj\Uhat)^{-1}\Uhat^{\top}\hatM(\ldotp,b,\ldotp)\Uhat$.
\item
Set $\Theta$ a $(K\times K)$ unitary matrix uniformly drawn and, $\forall x\in\X$, ${\Chat}(x):=\sum_{b=1}^{M}(\Uhat\Theta)(b,x)\Bhat(b)$.
\item
Compute $\widehat\Rbf$ a $(K\times K)$ unit Euclidean norm columns matrix that diagonalizes the matrix $\Chat(1)$:
\[
\widehat\Rbf^{-1}\Chat(1)\widehat\Rbf=\D{(\widehat\Lambda(1,1),\ldots,\widehat\Lambda(1,K))}\,.
\]
\item
For all $x,x'\in\X$, set $\widehat\Lambda(x,x'):=(\widehat\Rbf^{-1}\Chat(x)\widehat\Rbf)(x',x')$ and $\hatO:=\widehat\U\Theta\widehat\Lambda$.
\item
Consider the estimator $(\fhat_{M,x})_{x\in\X}$ defined by, for all $x\in\X$, $\fhat_{M,x} \eqdef \sum_{m=1}^{M}\hatO(m,x)\varphi_{{m}}$.
\item 
Set $\tilde{\pi} \eqdef \big(\Uhat^{\top}\hatO\big)^{-1}\Uhat^{\top}\hatl$.
\item
Consider the transition matrix estimator $\Qhat \eqdef \Pi_{\mathrm{TM}}\Big(\big(\Uhat^{\top}\hatO\D{\tilde{\pi}}\big)^{-1}\Uhat^{\top}\hatN\Uhat\big(\hatO^{\top}\Uhat\big)^{-1}\Big)$ where $\Pi_{\mathrm{TM}}$ denotes the projection (with respect to the scalar product given by the Frobenius norm) onto the convex set of transition matrices, and define
$\pihat$ as the stationary distribution of $\Qhat$.
\end{enumerate}
\caption{Nonparametric spectral estimation of the transition matrix and the emission laws}
\label{alg:Spectral}
\end{algorithm}

\subsection{Variance of the spectral estimators}
\label{subsec:spectral}
This section displays the results which allow to derive the asymptotic properties of the spectral estimators. The aim of Theorem~\ref{thm:spectral} is to provide an upper bound for the variance term with an explicit dependency with respect to both $p$ and $M$. The way it depends in $M$ is {described} by the quantity $\eta_{3}$ defined in \eqref{eq:def:eta3}. Recall that, in the examples {\bf (Spline)}, {\bf (Trig.)} and {\bf (Wav.)}, we have $\eta_{3}(\Phi_{M})\leq C_{\eta}M^{3/2}$ with $C_{\eta}>0$ a constant. 
In this section, assumption {\bf [H\ref{assum:Qstar}]} may be replaced by the following  weaker assumption {\bf [H\ref{assum:Qstarprime}']}.
\begin{hypHprime}
\label{assum:Qstarprime}
\begin{enumerate}[a)]
\item \label{assum:Qstar:rankprime}The transition matrix $\Qstar$ has full rank.
\item \label{assum:Qstar:deltastarprime}$(X_n)_{n\ge 1}$ is irreducible and aperiodic.
\end{enumerate}
\end{hypHprime}
Note that under {\bf [H\ref{assum:Qstarprime}']} and {\bf [H\ref{assum:stationary}]}, there exists $\pistarmin>0$ such that, for all $x\in\X$,
\begin{equation}
\label{eq:pistarmin}
\pistar_x\ge\pistarmin\eqsp.
\end{equation}

\begin{theorem}[Spectral estimators]
\label{thm:spectral}
Assume that {\bf [H\ref{assum:Qstarprime}']} and {\bf [H\ref{assum:stationary}]-[H\ref{assum:linearly:ind}]} hold. Assume also that for all $x\in\X$, $\fstar_x\in\mathrm{L}^{2}(\Y,\L)$. Then, there exist positive constant ${u}(\Q^{\star})$, ${\mathcal C}(\Q^{\star},\F^{\star})$ and ${\mathbf N}(\Q^{\star},\F^{\star})$ such that for any $u\geq{u}(\Q^{\star})$, any $\delta\in(0,1)$,  any $M\geq M_{\F^{\star}}$, there exists a permutation $\tau_{M}\in\mathcal S_{K}$ such that the spectral method estimators $\fhat_{M,x}$, $\pihat$ and $\Qhat$ (see Algorithm~\ref{alg:Spectral}) satisfy, for any $p\geq {\mathbf N}(\Q^{\star},\F^{\star})\eta_3(\Phi_{M})^{2}u(-\log\delta)/\delta^{2}$, with probability greater than $1-2\delta-4e^{-u}$, 
\eq
\notag
\max_{x\in\X}\lVert f^{\star}_{M,x}-\hat f_{M,\tau_{M}(x)}\lVert_{2}\leq {\mathcal C}(\Q^{\star},\F^{\star})\frac{\sqrt{-\log\delta}}{\delta}\frac{\eta_3(\Phi_{M})}{\sqrt p}\sqrt u\,,
\qe
\eq
\notag
\lVert\pi^{\star}-\mathbb P_{\tau_{M}}\hat\pi\lVert_{2}
\leq
 {\mathcal C}(\Q^{\star},\F^{\star})\frac{\sqrt{-\log\delta}}{\delta}\frac{\eta_3(\Phi_{M})}{\sqrt p}\sqrt u\,,
\qe
\eq
\notag
\lVert\Q^{\star}-\mathbb P_{\tau_{M}}\hat\Q \mathbb P_{\tau_{M}}^{\top}\lVert
\leq
 {\mathcal C}(\Q^{\star},\F^{\star})\frac{\sqrt{-\log\delta}}{\delta}\frac{\eta_3(\Phi_{M})}{\sqrt p}\sqrt u\,.
\qe
\end{theorem}

\begin{corollary}
\label{cor:spectral}
Assume that {\bf [H\ref{assum:Qstarprime}']} and {\bf [H\ref{assum:stationary}]-[H\ref{assum:linearly:ind}]} hold. Assume also that for all $x\in\X$, $\fstar_x\in\mathrm{L}^{2}(\Y,\L)$. Let $M_{p}$ be a sequence of integers tending to infinity and such that $\eta_3(\Phi_{M_p})=o(\sqrt{p/\log p})$. For each $p$, define $\fhat$,
$\Qhat$ and $\pihat$ as the estimators obtained by the spectral algorithm with this choice of $M_p$. Then, there exists a sequence of permutations $\tau_{p}\in\mathcal S_{K}$ such that
\[
\E\big[\max_{x\in\X}\lVert f^{\star}_{M_{p},x}-\hat f_{\tau_{p}(x)}\lVert_{2}\big]\ \vee\ 
\E\big[ \lVert \Q^{\star}-\mathbb P_{\tau_{p}}\hat\Q\mathbb P_{\tau_{p}}^{\top}\lVert\big]\ \vee\ 
\E\big[ \lVert\pi^{\star}-\mathbb P_{\tau_{p}}\hat\pi\lVert_{2}\big]
=O\big(\eta_3(\Phi_{M_p})\sqrt {\log p/p}\big)=o(1).
\]
Here, the expectations are with respect to the observations and to the random unitary matrix drawn at {\bf [Step 4]} of  Algorithm~\ref{alg:Spectral}.
\end{corollary}

\begin{proof}
Apply Theorem \ref{thm:spectral} where, for each $p$, we define $\delta_{p}$ such that $(-\log \delta_{p})/\delta_{p}^{2}\eqdef\log p$. $\delta_{p}$ goes to $0$ and $M_{p}$ goes to infinity as $p$ tends to infinity so that for any large enough $p$, $M_{p}\geq M_{\F^{\star}}$.  Let $\tau_{p}$ the permutation $\tau_{M_{p}}$ given by Theorem~\ref{thm:spectral}. Then, for all $p/({\mathbf N}(\Q^{\star},\F^{\star})\eta_3(\Phi_{M_p})^{2} \log p)\geq u\geq u(\Q^{\star})$, with probability $1-4e^{-u}-2\delta_{p}$,
\[
\max_{x\in\X}\lVert f^{\star}_{M,x}-\hat f_{M,\tau_{M}(x)}\lVert_{2}\vee\lVert\pi^{\star}-\mathbb P_{\tau_{p}}\hat\pi\lVert_{2}\vee\lVert\Q^{\star}-\mathbb P_{\tau_{p}}\hat\Q \mathbb P_{\tau_{p}}^{\top}\lVert
\leq
 {\mathcal C}(\Q^{\star},\F^{\star})\eta_3(\Phi_{M_p})\sqrt{{\log p}/{p}}\sqrt{u}\,.
\]
It yields
\begin{align*}
\limsup_{p\rightarrow +\infty}\E&\left[ \frac{ p}{\eta_3(\Phi_{M_p})^2\log p}\lVert \Q^{\star}-\mathbb P_{\tau_{p}}\hat\Q\mathbb P_{\tau_{p}}^{\top}\lVert^{2}\right] \\
&\leq 
{\mathcal C}(\Q^{\star},\F^{\star})^{2}\int_{0}^{+\infty}\limsup_{p\rightarrow +\infty}\P\left( \frac{\sqrt p}{{\mathcal C}(\Q^{\star},\F^{\star})\eta_3(\Phi_{M_p})\sqrt{\log p} }\lVert \Q^{\star}-\mathbb P_{\tau_{p}}\hat\Q\mathbb P_{\tau_{p}}^{\top}\lVert\geq  \sqrt{u}\right)\rmd u \\
&\leq 
{\mathcal C}(\Q^{\star},\F^{\star})^{2} u(\Q^{\star})+{\mathcal C}(\Q^{\star},\F^{\star})^{2}\int_{x(\Q^{\star})}^{+\infty}4e^{-u}\rmd u <+\infty\,.
\end{align*}
The proof is similar for the other terms.
\end{proof}

\section{Experimental results}
\label{sec:exp}
We have run several numerical experiments to assess the efficiency of our method. We consider $K=2$ emission laws of beta distributions with parameters $(2,5)$ and $(4,3)$. In all our experiments, the transition matrix $\Qstar$ is given by 
\[
\Qstar \eqdef \begin{pmatrix}0.4 & 0.6\\ 0.8 & 0.2 \end{pmatrix}\eqsp.
\]
We observe a sequence of $n=6\times 10^{4}$ variables $(Y_{i})_{i=1}^{n}$. As projection basis, we have considered the histogram basis or the trigonometric basis. The minimax adaptive procedure described in \cite{decastro:gassiat:lacour:2015} gives an estimation of $\Qstar$ and of the emission laws. Using the slope heuristic \cite{BMM12}, we find that the selected size of the model is $\hat M=11$ in the histogram case and $\hat M=13$ in the trigonometric case. Figure \ref{fig:trigoEstimationEmission} presents the adaptive estimation of the emission laws. From these estimates, we compute an estimation of the marginal smoothing probabilities using the forward-backward algorithm. The results are presented in Figure~\ref{fig:MarginalSmoothing}.
\begin{figure}
\begin{center}
\includegraphics[width=0.47\textwidth]{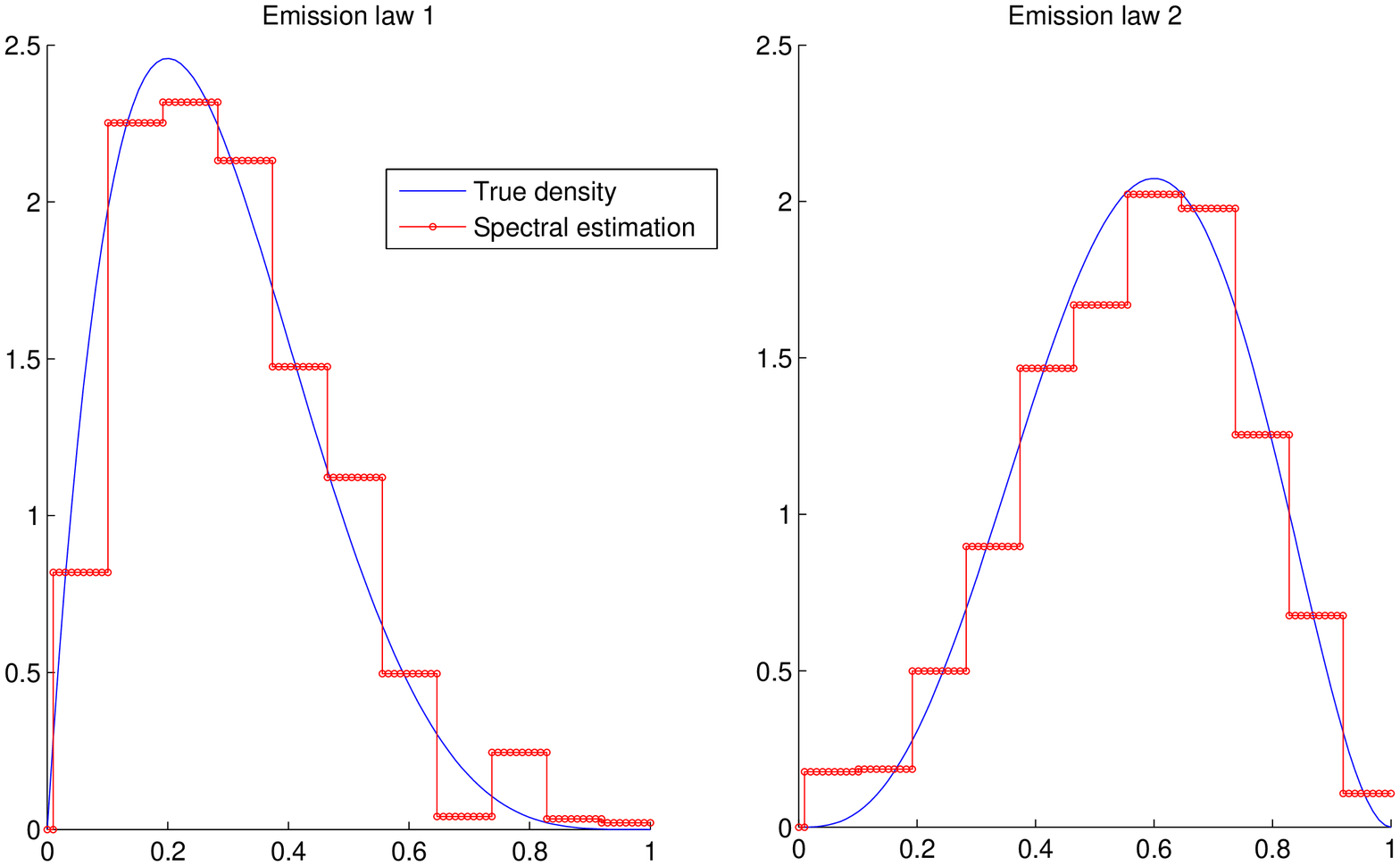}\hspace*{0.05\textwidth}
\includegraphics[width=0.47\textwidth]{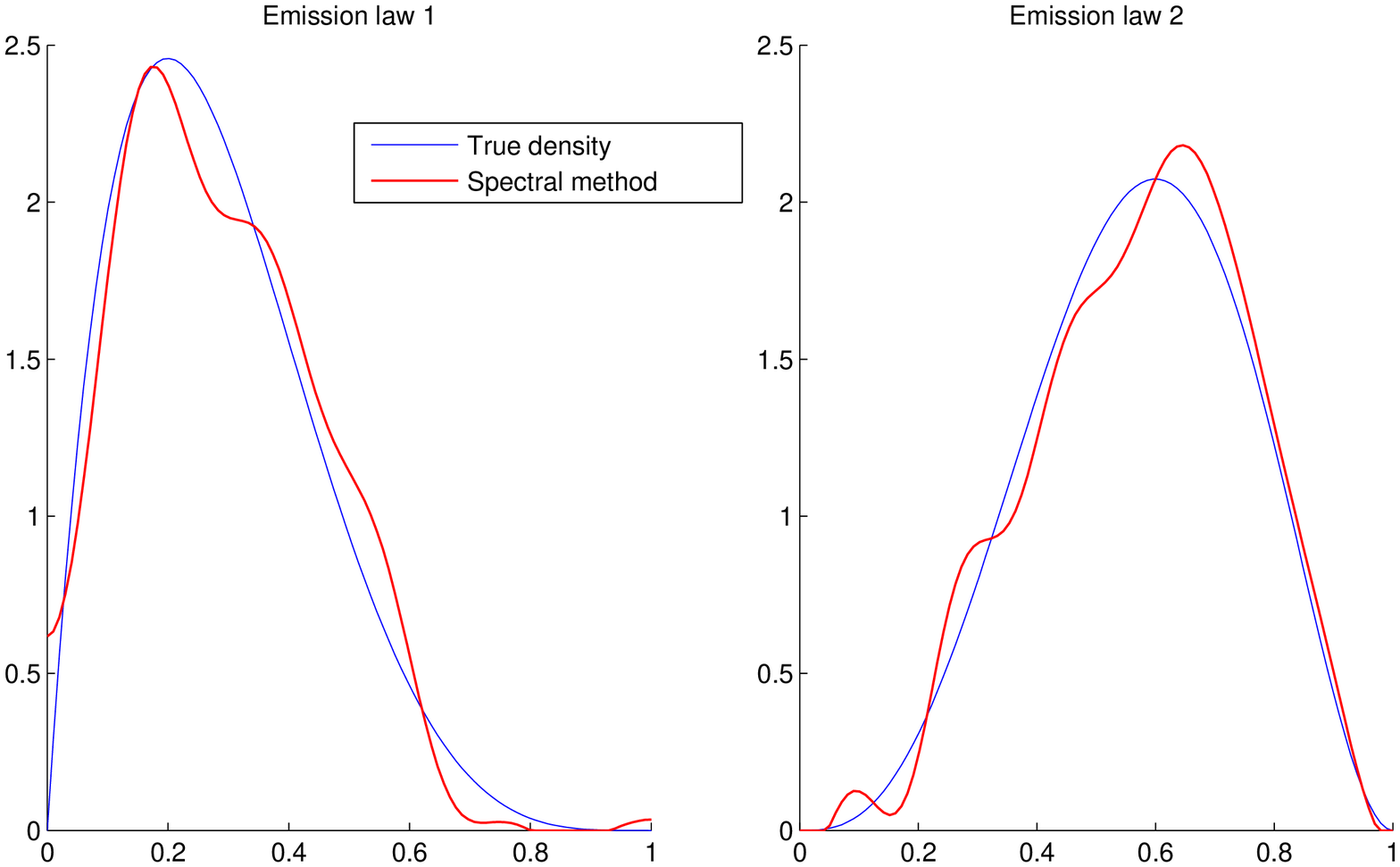}
\end{center}
\caption{Estimation of emission laws of beta distributions with parameters $(2,5)$ and $(4,3)$ using the spectral method. The projection basis is the histogram basis (left panel) or the trigonometric basis (right panel).}
\label{fig:trigoEstimationEmission}
\end{figure}

\begin{figure}
\begin{center}
\includegraphics[width=0.45\textwidth]{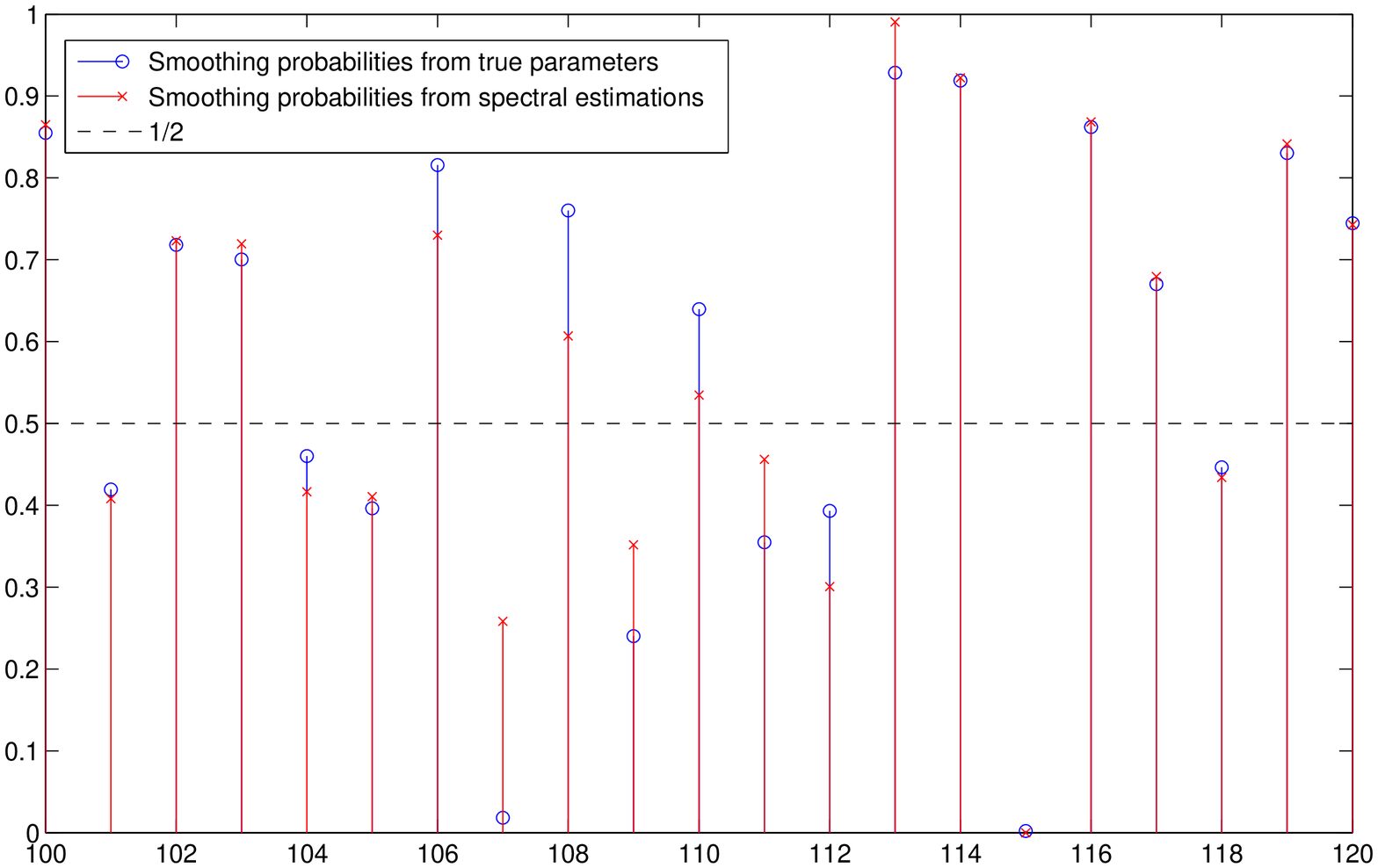}
\includegraphics[width=0.45\textwidth]{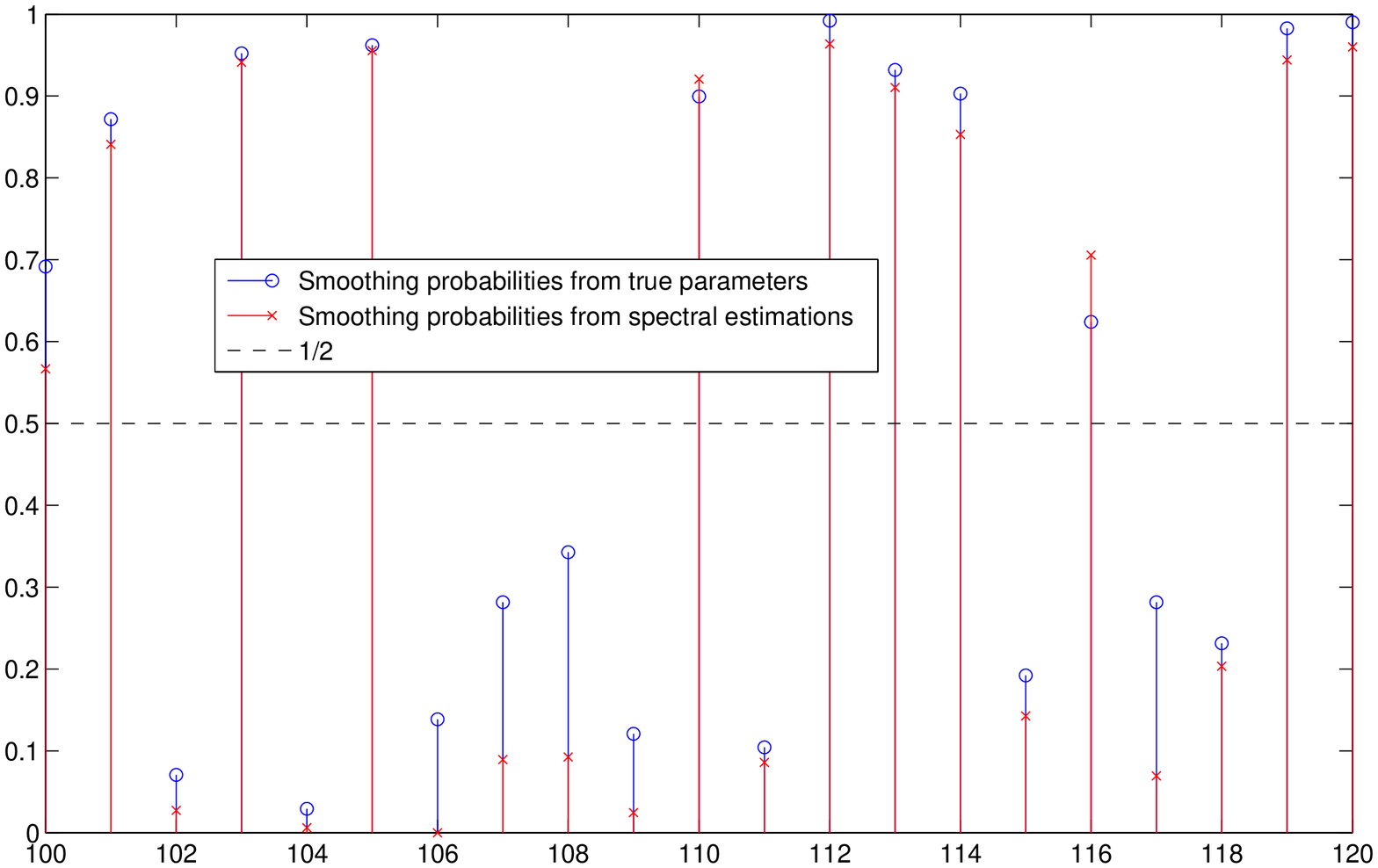}
\end{center}
\caption{Marginal smoothing probabilities obtained with the forward-backward algorithm combined with the spectral method using projection of the emission laws on the histogram basis (top panel) or the trigonometric basis (bottom panel).}
\label{fig:MarginalSmoothing}
\end{figure}


 
\appendix

\section{Control of the filtering error - Proof of Proposition~\ref{prop:filtering}}
\label{sec:filtResults}
Let $y_{1:n}\in\Y^n$. The aim of this section consists in establishing that the total variation error between $\filtstar{k}(\cdot,y_{1:n})$ and its approximations based on $\Qhat$ and $\fhat$ is bounded uniformly in time $k$. Before stating the main result, we introduce a standard decomposition of the filtering error $\filtstar{k}(\cdot,y_{1:k}) -\filthat{k}(\cdot,y_{1:k})$. For all $k\ge 1$, let $\Fstar_{k,y_k}$ be the forward kernel at time $k$ and $\Fhat_{k,y_k}$ its approximation, defined, for all $\nu\in\mathcal{P}(\X)$, as:
\[
\Fstar_{k,y_k}\nu(x) \eqdef \frac{\sum_{x'\in\X}\Qstar(x',x)\fstar_x(y_k)\nu(x')}{\sum_{x',x''\in\X}\Qstar(x',x'')\fstar_{x''}(y_k)\nu(x')}\,,
\]
and
\[
\Fhat_{k,y_k}\nu(x) \eqdef \frac{\sum_{x'\in\X}\Qhat(x',x)\fhat_x(y_k)\nu(x')}{\sum_{x',x''\in\X}\Qhat(x',x'')\fhat_{x''}(y_k)\nu(x')}\eqsp.
\]
Clearly, for all $y_{1:n}\in\Y^n$ and $2\le k \le n$, $\filtstar{k}(\cdot,y_{1:k}) = \Fstar_{k,y_k}\filtstar{k-1}(\cdot,y_{1:k-1})$ and $\filthat{k} (\cdot,y_{1:k})= \Fhat_{k,y_k}\filthat{k-1}(\cdot,y_{1:k-1})$. The filtering error is usually written as a sum of one step errors. For all $k\ge 2$,
\begin{align}
\filtstar{k}(\cdot,y_{1:k}) -\filthat{k}(\cdot,y_{1:k}) &= \Fstar_{k,y_k}\filtstar{k-1}(\cdot,y_{1:k-1}) - \Fhat_{k,y_k}\filthat{k-1}(\cdot,y_{1:k-1})\nonumber\\ 
&= \sum_{\ell=1}^{k-1}\Delta_{k,\ell}(y_{\ell:k}) + \Fstar_{k,y_k}\filthat{k-1}(\cdot,y_{1:k-1}) - \Fhat_{k,y_k}\filthat{k-1}(\cdot,y_{1:k-1})\eqsp,\label{eq:error:filt}
\end{align}
with $\Fstar_{1,y_1}\filthat{0} = \filtstar{1}(\cdot,y_1)$ and 
\[
\Delta_{k,\ell}(y_{\ell:k}) \eqdef \Fstar_{k,y_k}\ldots\Fstar_{\ell+1,y_{\ell+1}}\Fstar_{\ell,y_{\ell}}\filthat{\ell-1}(\cdot,y_{1:\ell-1}) -\Fstar_{k,y_k}\ldots\Fstar_{\ell+1,y_{\ell+1}}\filthat{\ell}(\cdot,y_{\ell})\eqsp.
\]
Let $\beta^{\star}_{\ell|k}[y_{\ell+1:k}]$ and $\Fstar_{\ell|k}[y_{\ell:k}]$ be the backward functions  and the forward smoothing transition matrix as defined in \cite[Chapter~3]{cappe:moulines:ryden:2005},
\begin{align}
\beta^{\star}_{\ell|k}[y_{\ell+1:k}](x_{\ell}) &\eqdef \sum_{x_{\ell+1:k}}\Qstar(x_{\ell},x_{\ell+1})\fstar_{x_{\ell+1}}(y_{\ell+1})\ldots \Qstar(x_{k-1},x_{k})\fstar_{x_{k}}(y_{k})\eqsp,\label{eq:def:beta}\\
\Fstar_{\ell|k}[y_{\ell:k}](x_{\ell-1},x_\ell) &\eqdef \frac{\beta^{\star}_{\ell|k}[y_{\ell+1:k}](x_{\ell})\Qstar(x_{\ell-1},x_{\ell})\fstar_{x_{\ell}}(y_{\ell})}{\sum_{x\in\X}\beta^{\star}_{\ell|k}[y_{\ell+1:k}](x)\Qstar(x_{\ell-1},x)\fstar_{x}(y_{\ell})}\label{eq:def:Fstar}\eqsp.
\end{align}
In the sequel, the dependency on the observations may be dropped to simplify notations. By \cite[Chapter~4]{cappe:moulines:ryden:2005}, for any probability distribution $\nu$, $\Fstar_k\ldots\Fstar_{\ell+1}\nu = \nu_{\ell|k}\Fstar_{\ell+1|k}\ldots\Fstar_{k|k}$, where $\nu_{\ell|k} \propto \beta^{\star}_{\ell|k}\nu$. Therefore, the filtering error \eqref{eq:error:filt} is given by:
\begin{equation}
\label{eq:filt:decomp}
\filtstar{k} -\filthat{k} = \sum_{\ell=1}^{k-1}\left(\mu^{\star}_{\ell|k}\Fstar_{\ell+1|k}\ldots\Fstar_{k|k} -\widehat{\mu}_{\ell|k}\Fstar_{\ell+1|k}\ldots\Fstar_{k|k}\right) + \Fstar_k\filthat{k-1} - \Fhat_k\filthat{k-1}\eqsp,
\end{equation}
where $\mu^{\star}_{\ell|k} \propto \beta^{\star}_{\ell|k}\Fstar_{\ell}\filthat{\ell-1}$ and $\widehat{\mu}_{\ell|k} \propto \beta^{\star}_{\ell|k}\filthat{\ell}$. By {\bf [H\ref{assum:Qstar}]-\ref{assum:Qstar:deltastar})}, the transition matrix $\Fstar_{k|n}$ can be lower bounded uniformly in its first component:
\[ 
\Fstar_{\ell|k}(x,x')\ge \frac{\dstar}{1-\dstar}\frac{\beta^{\star}_{\ell|k}[y_{\ell+1:k}](x')\fstar_{x'}(y_{\ell})}{\sum_{z\in\X}\beta^{\star}_{\ell|k}[y_{\ell+1:k}](z)\fstar_{z}(y_{\ell})}\eqsp.
\]
By \cite[Chapter~4]{cappe:moulines:ryden:2005}, this allows to write, 
\begin{equation}
\label{eq:forget:filt}
\left\|\mu^{\star}_{\ell|k}\Fstar_{\ell+1|k}\ldots\Fstar_{k|k} -\widehat{\mu}_{\ell|k}\Fstar_{\ell+1|k}\ldots\Fstar_{k|k}\right\|_{\mathrm{tv}}\le \rho_{\star}^{k-\ell}\|\mu^{\star}_{\ell|k}-\widehat{\mu}_{\ell|k}\|_{\mathrm{tv}}\eqsp.
\end{equation}
Eq. \eqref{eq:forget:filt} is the crucial step to obtain the upper bound for the filtering error stated in Proposition~\ref{prop:filtering}.
By \eqref{eq:filt:decomp} and \eqref{eq:forget:filt},
\[
\|\filtstar{k} -\filthat{k}\|_{\mathrm{tv}}\le \sum_{\ell=1}^{k-1}\rho_{\star}^{k-\ell}\left\|\mu^{\star}_{\ell|k}-\widehat{\mu}_{\ell|k}\right\|_{\mathrm{tv}} + \left\|\Fstar_k\filthat{k-1} - \Fhat_k\filthat{k-1}\right\|_{\mathrm{tv}}\eqsp.
\]
For all $1\le \ell\le k-1$ and all bounded function $h$ on $\X$, $\left|\mu^{\star}_{\ell|k}(h)-\widehat{\mu}_{\ell|k}(h)\right|\le T_1 + T_2$ where
\begin{align*}
T_1 &\eqdef \left|\frac{\sum_{x\in \X}\beta^{\star}_{\ell|k}[y_{\ell+1:k}](x)h(x)\left[\Fstar_{\ell}\filthat{\ell-1}(x)-\Fhat_{\ell}\filthat{\ell-1}(x)\right]}{\sum_{x\in \X}\beta^{\star}_{\ell|k}[y_{\ell+1:k}](x)\Fstar_{\ell}\filthat{\ell-1}(x)}\right| \eqsp,\\
T_2 &\eqdef  \left|\frac{\sum_{x\in \X}\beta^{\star}_{\ell|k}[y_{\ell+1:k}](x)h(x)\Fhat_{\ell}\filthat{\ell-1}(x)}{\sum_{x\in \X}\beta^{\star}_{\ell|k}[y_{\ell+1:k}](x)\Fhat_{\ell}\filthat{\ell-1}(x)}\right| \cdot \left|\frac{\sum_{x\in \X}\beta^{\star}_{\ell|k}[y_{\ell+1:k}](x)\left[\Fstar_{\ell}\filthat{\ell-1}(x)-\Fhat_{\ell}\filthat{\ell-1}(x)\right]}{\sum_{x\in \X}\beta^{\star}_{\ell|k}[y_{\ell+1:k}](x)\Fstar_{\ell}\filthat{\ell-1}(x)}\right| \eqsp.
\end{align*}
Both $T_1$ and $T_2$ are upper bounded by the same term so that
\[
T_1+T_2 \le 2\frac{\|h\|_{\infty}\cdot\|\beta^{\star}_{\ell|k}[y_{\ell+1:k}]\|_{\infty}}{\inf_{x\in\X} \beta^{\star}_{\ell|k}[y_{\ell+1:k}](x)}\|\Fstar_{\ell}\filthat{\ell-1}-\Fhat_{\ell}\filthat{\ell-1}\|_{\mathrm{tv}}\eqsp.
\]
By \eqref{eq:def:beta}, for all $x\in\X$, $\beta^{\star}_{\ell|k}[y_{\ell+1:k}](x)\le (1-\dstar)\sum_{x_{k+1:n}}\fstar_{x_{k+1}}(y_{k+1})\ldots \Qstar(x_{n-1},x_{n})\fstar_{x_{n}}(y_{n})$ and $\beta^{\star}_{\ell|k}[y_{\ell+1:k}](x)\ge \dstar\sum_{x_{k+1:n}}\fstar_{x_{k+1}}(y_{k+1})\ldots \Qstar(x_{n-1},x_{n})\fstar_{x_{n}}(y_{n})$, showing that
\[
T_1+T_2 \le 2\|h\|_{\infty}\left(\frac{1-\dstar}{\dstar}\right)\|\Fstar_{\ell}\filthat{\ell-1}-\Fhat_{\ell}\filthat{\ell-1}\|_{\mathrm{tv}}\eqsp.
\]
Now, for all $2\le \ell\le k$ and all bounded function $h$ on $\X$, $\left|\Fstar_{\ell}\filthat{\ell-1}(h)-\Fhat_{\ell}\filthat{\ell-1}(h)\right|\le R_1 + R_2$, where
\begin{align*}
R_1 &\eqdef \left|\frac{\sum_{x,x'\in\X}\filthat{\ell-1}(x)\left[\Qstar(x,x')\fstar_{x'}(y_{\ell})-\Qhat(x,x')\fhat_{x'}(y_{\ell})\right]h(x')}{\sum_{x,x'\in\X}\filthat{\ell-1}(x)\Qstar(x,x')\fstar_{x'}(y_{\ell})}\right| \eqsp,\\
R_2 &\eqdef  \left|\frac{\sum_{x,x'\in\X}\filthat{\ell-1}(x)\Qhat(x,x')\fhat_{x'}(y_{\ell})h(x')}{\sum_{x,x'\in\X}\filthat{\ell-1}(x)\Qhat(x,x')\fhat_{x'}(y_{\ell})}\right|\\
&\hspace{5cm}\times\left|\frac{\sum_{x,x'\in\X}\filthat{\ell-1}(x)\left[\Qstar(x,x')\fstar_{x'}(y_{\ell})-\Qhat(x,x')\fhat_{x'}(y_{\ell})\right]}{\sum_{x,x'\in\X}\filthat{\ell-1}(x)\Qstar(x,x')\fstar_{x'}(y_{\ell})}\right| \eqsp.
\end{align*}
Then, 
\begin{align*}
R_1&\le \left(\sum_{x,x'\in\X}\filthat{\ell-1}(x)\Qstar(x,x')\fstar_{x'}(y_{\ell})\right)^{-1}\sum_{x,x'\in\X}\filthat{\ell-1}(x)\left|\Qstar(x,x')\fstar_{x'}(y_{\ell})-\Qhat(x,x')\fhat_{x'}(y_{\ell})\right|h(x')\eqsp,\\
&\le \left(\sum_{x,x'\in\X}\filthat{\ell-1}(x)\Qstar(x,x')\fstar_{x'}(y_{\ell})\right)^{-1}\sum_{x,x'\in\X}\filthat{\ell-1}(x)\left|\Qstar(x,x')-\Qhat(x,x')\right|\fstar_{x'}(y_{\ell})h(x') \\
&\hspace{1.5cm}+ \left(\sum_{x,x'\in\X}\filthat{\ell-1}(x)\Qstar(x,x')\fstar_{x'}(y_{\ell})\right)^{-1}\sum_{x,x'\in\X}\filthat{\ell-1}(x)\Qhat(x,x')\left|\fstar_{x'}(y_{\ell})-\fhat_{x'}(y_{\ell})\right|h(x')\eqsp,\\
& \le \|h\|_{\infty}\left[\|\Qstar-\Qhat\|_{F}/\delta^{\star} + c_{\star}^{-1}(y_{\ell})\max_{x\in\X}\left|f^{\star}_{x}(y_{\ell})-\hat{f}_{x}(y_{\ell})\right|\right]\eqsp,
\end{align*}
where $c_{\star}$ is defined in \eqref{eq:def:cstar}. The same upper bound holds for $R_2$. In the case $\ell=1$, 
\[
\left\|\Fstar_{1}\filthat{0}-\filthat{1}\right\|_{\mathrm{tv}}\le \left\|\filtstar{1}-\filthat{1}\right\|_{\mathrm{tv}}\le2\left[\left\|\pistar-\pihat\right\|_{2}/\delta^{\star} + c_{\star}^{-1}(y_{1})\max_{x\in\X}\left|f^{\star}_{x}(y_{1})-\hat{f}_{x}(y_{1})\right| \right]\eqsp.
\]
Therefore, the filtering error is upper bounded as follows:
\begin{multline*}
\|\filtstar{k} -\filthat{k}\|_{\mathrm{tv}}\le 4\left(\frac{1-\dstar}{\dstar}\right) \sum_{\ell=2}^{k}\rho_{\star}^{k-\ell}\left[\|\Qstar-\Qhat\|_{F}/\delta^{\star} +c_{\star}^{-1}(y_{\ell})\max_{x\in\X}\left|f^{\star}_{x}(y_{\ell})-\hat{f}_{x}(y_{\ell})\right|\right]\\
+4\left(\frac{1-\dstar}{\dstar}\right)\rho_{\star}^{k-1}\left[\left\|\pistar-\pihat\right\|_{2}/\delta^{\star} + c_{\star}^{-1}(y_1)\max_{x\in\X}\left|f^{\star}_{x}(y_{1})-\hat{f}_{x}(y_{1})\right| \right] \eqsp.
\end{multline*}

\section{Control of the marginal smoothing error - Proof of Proposition~\ref{prop:smoothing}}
\label{sec:smoothResults}
Let $y_{1:n}\in\Y^n$. The aim of this section consists in establishing that the total variation error between $\filtstar{k|n}(\cdot,y_{1:n})$ and its approximations based on $\Qhat$ and $\fhat$ is bounded uniformly in time $k$. Before stating the main result, we display the decomposition of the smoothing error $\filtstar{k|n}(\cdot,y_{1:n}) -\filthat{k|n}(\cdot,y_{1:n})$ depicted in \cite{douc:garivier:moulines:olsson:2011} and used in \cite{dubarry:lecorff:2013} to obtain nonasymptotic upper bounds for the marginal smoothing error when $\filtstar{k|n}(\cdot,y_{1:n})$ is approximated using Sequential Monte Carlo methods. In the sequel, the dependency on the observations may be dropped to simplify notations. For any bounded function $h$ on $\X^{n}$, $\smoothstar{1:n}{n}(h)$ can be written, for any $1\le \ell\le n$
\[
\smoothstar{1:n}{n}(h) = \frac{\smoothstar{1:\ell}{\ell}(L^{\star}_{\ell,n}(\cdot,h))}{\smoothstar{1:\ell}{\ell}(L^{\star}_{\ell,n}(\cdot,\1))}\eqsp,
\]
where $\1$ is the constant function which equals 1 and, for all $x_{1:\ell}\in\X^{\ell}$,
\begin{equation}
\label{eq:Lstar}
L^{\star}_{\ell,n}(x_{1:\ell},h) \eqdef \sum_{x_{\ell+1:n}\in\X^{n-\ell}}\prod_{u=\ell+1}^n\Qstar(x_{u-1},x_u)\fstar_{x_u}(y_u)h(x_{1:n})\eqsp.
\end{equation}
As for the filtering error, the smoothing error can be decomposed as a telescopic sum of one step errors:
\begin{multline}
\smoothhat{1:n}{n}(h) - \smoothstar{1:n}{n}(h) = \sum_{\ell=2}^n\left(\frac{\smoothhat{1:\ell}{\ell}(L^{\star}_{\ell,n}(\cdot,h))}{\smoothhat{1:\ell}{\ell}(L^{\star}_{\ell,n}(\cdot,\1))} - \frac{\smoothhat{1:\ell-1}{\ell-1}(L^{\star}_{\ell-1,n}(\cdot,h))}{\smoothhat{1:\ell-1}{\ell-1}(L^{\star}_{\ell-1,n}(\cdot,\1))}\right)\\
+ \frac{\filthat{1}(L^{\star}_{1,n}(\cdot,h))}{\filthat{1}(L^{\star}_{1,n}(\cdot,\1))} - \frac{\filtstar{1}(L^{\star}_{1,n}(\cdot,h))}{\filtstar{1}(L^{\star}_{1,n}(\cdot,\1))}\eqsp.\label{eq:telescopic:smooth}
\end{multline}
This smoothing error can be written using filtering distributions only by introducing the following backward operators:
\begin{align*}
\mathcal{L}^{\star}_{\ell,n}(x_\ell,h) &\eqdef \sum_{x_{1:\ell-1}} B^{\star}_{\filtstar{\ell-1}}(x_\ell,x_{\ell-1})\ldots B^{\star}_{\filtstar{1}}(x_2,x_{1})L^{\star}_{\ell,n}(x_{1:\ell},h)\eqsp,\\
\widehat{\mathcal{L}}_{\ell,n}(x_\ell,h) &\eqdef \sum_{x_{1:\ell-1}} \widehat{B}_{\filthat{\ell-1}}(x_\ell,x_{\ell-1})\ldots \widehat{B}_{\filthat{1}}(x_2,x_{1})L^{\star}_{\ell,n}(x_{1:\ell},h)\eqsp,
\end{align*}
where for all  $\nu\in\mathcal P(\X)$, $B_{\nu}$ is the backward smoothing kernel given by 
\[
B^{\star}_{\nu}(x,x')\eqdef \frac{\Qstar(x',x)\nu(x')}{\sum_{z\in\X}\Qstar(z,x)\nu(z)}\eqsp.
\]
Then, for all $2\le t\le n$, the one step error at time $\ell$ is given by
\begin{equation}
\label{eq:smooth:via:filt}
\delta_{\ell,n}(h)\eqdef\frac{\smoothhat{1:\ell}{\ell}(L^{\star}_{\ell,n}(\cdot,h))}{\smoothhat{1:\ell}{\ell}(L^{\star}_{\ell,n}(\cdot,\1))} - \frac{\smoothhat{1:\ell}{\ell}(L^{\star}_{\ell,n}(\cdot,h))}{\smoothhat{1:\ell}{\ell}(L^{\star}_{\ell,n}(\cdot,\1))} = \frac{\filthat{\ell}(\widehat{\mathcal{L}}_{\ell,n}(\cdot,h))}{\filthat{\ell}(\widehat{\mathcal{L}}_{\ell,n}(\cdot,\1))} - \frac{\filthat{\ell-1}(\widehat{\mathcal{L}}_{\ell-1,n}(\cdot,h))}{\filthat{\ell-1}(\widehat{\mathcal{L}}_{\ell-1,n}(\cdot,\1))}\eqsp.
\end{equation}
This decomposition allows to obtain the upper bound for the marginal smoothing error stated in Proposition~\ref{prop:smoothing}. 
The result is obtained by applying the decompositions \eqref{eq:telescopic:smooth} and \eqref{eq:smooth:via:filt} to a bounded function $h$ on $\X^n$ which depends on $x_k$ only: for all $(x_1,\ldots,x_n)\in\X^n$, $h(x_1,\ldots,x_n) = h(x_k)$. The one step error given by \eqref{eq:smooth:via:filt} is then analyzed separately wether $k\ge \ell$ or $k<\ell$.

\medskip

\noindent {\textbf{Case} $k\ge \ell$}\\
In this case, the function $L^{\star}_{\ell,n}(\cdot,h)$ defined in \eqref{eq:Lstar} depends on $x_\ell$ only. Therefore, $\widehat{\mathcal{L}}_{\ell,n}(x_\ell,h) = L^{\star}_{\ell,n}(x_\ell,h) = \mathcal{L}^{\star}_{\ell,n}(x_\ell,h)$. Thus, $\widehat{\mathcal{L}}_{\ell-1,n}(x_{\ell-1},h) = \sum_{x_\ell\in\X}\Qstar(x_{\ell-1},x_\ell)\fstar_{x_\ell}(y_\ell)\mathcal{L}^{\star}_{\ell,n}(x_\ell,h)$ and the one step error given by \eqref{eq:smooth:via:filt} becomes
\[
\delta_{\ell,n}(h) = \frac{\filthat{\ell}(\mathcal{L}^{\star}_{\ell,n}(\cdot,h))}{\filthat{\ell}(\mathcal{L}^{\star}_{\ell,n}(\cdot,\1))} - \frac{\filthat{\ell-1}(\sum_{x_\ell\in\X}\Qstar(\cdot,x_\ell)\fstar_{x_\ell}(y_\ell)\mathcal{L}^{\star}_{\ell,n}(x_\ell,h))}{\filthat{\ell-1}(\sum_{x_\ell\in\X}\Qstar(\cdot,x_\ell)\fstar_{x_\ell}(y_\ell)\mathcal{L}^{\star}_{\ell,n}(x_\ell,\1))}\eqsp.
\]
Define the measures $\mu_\ell$ and  $\widehat{\mu}_\ell$ on $\X$ by $\mu_\ell(x_\ell) \eqdef \sum_{x_{\ell-1}\in\X}\filthat{\ell-1}(x_{\ell-1})\Qstar(x_{\ell-1},x_\ell)\fstar_{x_\ell}(y_\ell)$ and $\widehat{\mu}_\ell(x_\ell) \eqdef \sum_{x_{\ell-1}\in\X}\filthat{\ell-1}(x_{\ell-1})\Qhat(x_{\ell-1},x_\ell)\fhat_{x_\ell}(y_\ell)$. Then,
\[
\delta_{\ell,n}(h) = \frac{\widehat{\mu}_\ell(\mathcal{L}^{\star}_{\ell,n}(\cdot,h))}{\widehat{\mu}_\ell(\mathcal{L}^{\star}_{\ell,n}(\cdot,\1))} - \frac{\mu_\ell(\mathcal{L}^{\star}_{\ell,n}(\cdot,h))}{\mu_\ell(\mathcal{L}^{\star}_{\ell,n}(\cdot,\1))}\eqsp.
\]
By \cite[Lemma~4.3.23]{cappe:moulines:ryden:2005} and {\bf [H\ref{assum:Qstar}]-\ref{assum:Qstar:deltastar})},
$\left|\delta_{\ell,n}(h)\right| \le \rho_{\star}^{k-\ell}(1-\dstar)\left\|\mu_\ell/\mu_\ell(\1)-\widehat{\mu}_\ell/\widehat{\mu}_\ell(\1)\right\|_{\mathrm{tv}}\|h\|_{\infty}/\dstar$.
Following the same steps as for the proof of Proposition~\ref{prop:filtering} yields 
\[
\left\|\mu_\ell/\mu_\ell(\1)-\widehat{\mu}_\ell/\widehat{\mu}_\ell(\1)\right\|_{\mathrm{tv}} \le 2\|\Qstar-\Qhat\|_{F}/\delta^{\star} + 2c^{-1}_{\star}(y_{\ell})\max_{x\in\X}\left|f^{\star}_{x}(y_{\ell})-\hat{f}_{x}(y_{\ell})\right|\eqsp.
\]
The term $\filthat{1}(L^{\star}_{1,n}(\cdot,h))/\filthat{1}(L^{\star}_{1,n}(\cdot,\1)) - \filtstar{1}(L^{\star}_{1,n}(\cdot,h))/\filtstar{1}(L^{\star}_{1,n}(\cdot,\1))$ is dealt with similarly.

\medskip

\noindent {\textbf{Case} $k<\ell$}\\
In this case, $L^{\star}_{\ell,n}(x_{1:\ell},h) = h(x_k)\mathcal{L}^{\star}_{\ell,n}(x_\ell,\1)$. Therefore, 
\begin{align*}
\widehat{\mathcal{L}}_{\ell,n}(x_\ell,h) &= \sum_{x_{1:\ell-1}} \widehat{B}_{\filthat{\ell-1}}(x_\ell,x_{\ell-1})\ldots \widehat{B}_{\filthat{1}}(x_2,x_{1})h(x_k)\mathcal{L}^{\star}_{\ell,n}(x_\ell,\1)\eqsp,\\
 &= \sum_{x_{k:\ell-1}} \mathcal{L}^{\star}_{\ell,n}(x_\ell,\1)\widehat{B}_{\filthat{\ell-1}}(x_\ell,x_{\ell-1})\ldots \widehat{B}_{\filthat{k}}(x_{k+1},x_{k})h(x_k)\eqsp.
\end{align*}
On the other hand, if $\nu_\ell(x_\ell) \eqdef  \sum_{x_{\ell-1}\in\X}\filthat{\ell-1}(x_{\ell-1})\Qstar(x_{\ell-1},x_\ell)\fstar_{x_\ell}(y_\ell)\mathcal{L}^{\star}_{\ell,n}(x_\ell,\1)$,
\[
\filthat{\ell-1}(\widehat{\mathcal{L}}_{\ell-1,n}(\cdot,h)) = \sum_{x_{k:\ell}\in\X^{\ell-k+1}}\nu_\ell(x_\ell)\widehat{B}_{\filthat{\ell-1}}(x_{\ell},x_{\ell-1})\ldots \widehat{B}_{\filthat{k}}(x_{k+1},x_{k})h(x_k)\eqsp.
\]
Define $\widehat{\nu}_\ell(x_\ell) \eqdef  \filthat{\ell}(x_\ell)\mathcal{L}^{\star}_{\ell,n}(x_\ell,\1) = \sum_{x_{\ell-1}\in\X}\filthat{\ell-1}(x_{\ell-1})\Qhat(x_{\ell-1},x_\ell)\fhat_{x_\ell}(y_\ell)\mathcal{L}^{\star}_{\ell,n}(x_\ell,\1)$. Then,  the one step error given by \eqref{eq:smooth:via:filt} becomes
\[
\delta_{\ell,n}(h) = \sum_{x_{k:\ell-1}} \left(\frac{\widehat{\nu}_\ell(x_\ell)}{\widehat{\nu}_\ell(\1)}-\frac{\nu_\ell(x_\ell)}{\nu_\ell(\1)}\right)\widehat{B}_{\filthat{\ell-1}}(x_\ell,x_{\ell-1})\ldots \widehat{B}_{\filthat{k}}(x_{k+1},x_{k})h(x_k)
\]
By \cite[Lemma~4.3.23]{cappe:moulines:ryden:2005} and the fact that, for all $(x,x')\in\X^2$, $\Qhat(x,x')\ge \widehat{\delta}$,
\[
\left|\delta_{\ell,n}(h)\right| \le \|h\|_{\infty}\widehat{\rho}^{\,\ell-k}\left\|\frac{\widehat{\nu}_\ell(\cdot)}{\widehat{\nu}_\ell(\1)} - \frac{\nu_\ell(\cdot)}{\nu_\ell(\1)}\right\|_{\mathrm{tv}}\eqsp.
\]
As for all $x_\ell\in\X$, $\mathcal{L}^{\star}_{\ell,n}(x_\ell,\1)/\|\mathcal{L}^{\star}_{\ell,n}(\cdot,\1)\|_{\infty}\ge \dstar/(1-\dstar)$, following the same steps as for the proof of Proposition~\ref{prop:filtering} yields 
\[
\left\|\frac{\widehat{\nu}_\ell(\cdot)}{\widehat{\nu}_\ell(\1)} - \frac{\nu_\ell(\cdot)}{\nu_\ell(\1)}\right\|_{\mathrm{tv}}\le2 \left(\frac{1-\dstar}{\dstar}\right)\left(\|\Qstar-\Qhat\|_{F}/\delta^{\star} + c_{\star}^{-1}(y_\ell)\max_{x\in\X}\left|f^{\star}_{x}(y_{\ell})-\hat{f}_{x}(y_{\ell})\right|\right)\eqsp.
\]

\section{Nonparametric spectral estimators}
Theorem~\ref{thm:spectral} follows from the following more precise results proved in this section. The proofs of the intermediate lemmas require assumptions {\bf [H\ref{assum:Qstarprime}']} and {\bf [H\ref{assum:stationary}]-[H\ref{assum:linearly:ind}]}.
\begin{lemma}
\label{lem:sigmastar}
There exist a constant $0<\sigmaFamily\leq1$ and a positive integer $M_{\frakFstar}$ such that for all $M\geq M_{\frakFstar}$,
\[
\sigma_{K}(\O)\geq{\sigmaFamily}>0\,.
\]
\end{lemma}

\begin{proof}
By {\bf [H\ref{assum:linearly:ind}]}, the $(K\times K)$ Gram matrix defined by $\Ostar^{\top}\Ostar:=(\langle \fstar_{x_{1}},\fstar_{x_{2}}\rangle)_{x_{1},x_{2}\in\X}$ is invertible. Let $\varepsilon_{\frakFstar,M}$ be given by:
\eq
\label{eq:OMerror}
\varepsilon_{\frakFstar,M}\eqdef \left\lVert \O^{\top}\O-\Ostar^{\top}\Ostar\right\lVert = \left\lVert(\langle  \fstar_{M,x_{1}}, \fstar_{M,x_{2}}\rangle-\langle \fstar_{x_{1}},\fstar_{x_{2}}\rangle)_{x_{1},x_{2}\in\X}\right\lVert\,.
\qe
From \eqref{eq:ConvergenceL2}, there exists $M_{\frakFstar}\geq1$ such that for all $M\geq M_{\F^{\star}}$, $\varepsilon_{\frakFstar,M}\leq 3\lambda_{K}(\Ostar^{\top}\Ostar)/4$.
By Weyl's inequality (see Theorem~\ref{thm:Weyl}), $\sigma_{K}^{2}(\O)=\lambda_{k}(\O^{\top}\O)\geq{\lambda_{K}(\Ostar^{\top}\Ostar)}/{4}$. If $\sigma_{K}(\Ostar)\eqdef\lambda^{1/2}_{K}(\Ostar^{\top}\Ostar)$, note that for all $M\geq M_{\F^{\star}}$, $\sigma_{K}(\O)\geq{\sigma_{K}(\Ostar)}/{2}$, which concludes the proof.
\end{proof}

\noindent Define the \textit{pseudo spectral gap} $\psg$ of the Markov chain $(X_{n})_{n\geq1}$ as
\[
\psg \eqdef \max_{k\geq1}\left\{\gap\left(\D{\pistar}^{-1}(\Qstar^{\top})^{k}\D{\pistar}\Qstar^{k}\right)/k\right\}\eqsp,
\]
where $\gap(A)$ denotes the spectral gap of a transition matrix $A$ defined by
\[
\gap(A)\eqdef\begin{cases}
    1-\max\{\lambda\ :\ \lambda\ \mathrm{eigenvalue\ of}\ A\,,\ \lambda\neq1\}  & \text{if eigenvalue 1 has multiplicity 1}, \\
   0   & \text{otherwise}.
\end{cases}
\]
Note that $\psg$ depends only on the transition matrix $\Qstar$ which is assumed to be aperiodic and irreducible with unique stationary distribution $\pi^{\star}$. Perron-Frobenius theorem ensures that the spectral gap $\gap(A)$ is well defined and such that $0\leq\gap(A)\leq2$.
\begin{remark}
If $\Qstar$ is aperiodic and irreducible then $\psg>0$. In this case, there exists $k$ such that $\Qstar^{k}$ is positive (entrywise) and so is $A\eqdef\D{\pistar}^{-1}(\Qstar^{\top})^{k}\D{\pistar}\Qstar^{k}$. As $A$ is a positive transition matrix, Perron-Frobenius theorem ensures that its spectral gap is positive. 
\end{remark}
\begin{remark}
 If $\Qstar$ is aperiodic, irreducible and reversible then $\psg=\gap(\Qstar)(2-\gap(\Qstar))>0$, see \cite{paulin2012concentration} and references therein.
\end{remark}
\noindent
Define the mixing time $\mixtime$ of the Markov chain $(X_{n})_{n\geq1}$ as
\[
\mixtime:=\frac{1+3\log 2-\log\pistarmin}{\psg}\,.
\]
This mixing time has a deeper interpretation in terms of convergence towards the stationary distribution in total variation norm, see \cite{paulin2012concentration} for instance. For any $\delta\in(0,1)$, set 
\begin{equation}
\label{eq:cstar}
\mathcal{C}_{\star}(\Qstar,\delta)\eqdef\sqrt{ 2/\psg}+2\sqrt{-2\,\mathbb T_{\mathrm{mix}}\log\delta}\eqsp,
\end{equation}
which is a constant that depends only on $\Qstar$ and $\delta$. 

\begin{theorem}
\label{thm:ThmMain1}
Assume that {\bf [H\ref{assum:Qstarprime}']} and {\bf [H\ref{assum:stationary}]-[H\ref{assum:linearly:ind}]} hold. Let $\delta,\delta'\in(0,1)$ then, with probability greater than $1-2\delta-4\delta'$, there exists a permutation $\tau\in\mathcal S_{K}$ such that the spectral method estimators $\fhat_{M,x}$, $\pihat$ and $\Qhat$ (see Algorithm~\ref{alg:Spectral} for a definition) satisfy, for any $M\geq M_{\F^{\star}}$,
\begin{enumerate}[-]
\item for all $p\geq\borneNfinale$ and all $x\in\X$,
\begin{equation}
\label{eq:ControldeOthm}
\lVert  \fstar_{M,x}-\fhat_{M,\tau(x)}\lVert_{2}\leq\mathcal {C}_{M}(\Qstar,\F^{\star},\delta)\concentration\eta_3(\Phi_{M})/\sqrt p\eqsp,
\end{equation}
\item for all $p\geq\borneNfinalebis$,
\begin{equation}
\label{eq:EstimationOfQ}
\lVert\Qstar-\mathbb \P_{\tau}\Qhat \mathbb P_{\tau}^{\top}\lVert
\leq
\mathcal {D}_{M}(\Qstar,\F^{\star},\delta)\concentration\eta_3(\Phi_{M})/\sqrt p\eqsp,
\end{equation}
\item for all $p\geq\borneNfinalebisbis$,
\begin{equation}
\label{eq:EstimationOfL}
\lVert\pistar-\mathbb P_{\tau}\pihat\lVert_{2}
\leq
\mathcal {E}_{M}(\Qstar,\F^{\star},\delta)\concentration\eta_3(\Phi_{M})/\sqrt p\eqsp,
\end{equation}
\end{enumerate}
where $\mathbb P_{\tau}$ is the permutation matrix associated to $\tau$, and
\begin{align*}
\borneNfinale&\eqdef\frac{4K}{3\sigmaFamily^2}\mathcal C_{M}(\Qstar,\F^{\star},\delta)^{2}\,\concentration^{2} \eta_3(\Phi_{M})^{2}\,,
\\
\borneNfinalebis&\eqdef\frac{4}{{\pistar}_{\mathrm{min}}^{2}}\mathcal D'_{M}(\Qstar,\F^{\star},\delta)^{2}\,\concentration^{2} \eta_3(\Phi_{M})^{2}\,,
\\
\borneNfinalebisbis&\eqdef\frac{4}{\sigma_{K}^{2}(\mathbf{A}_{\Qstar})}\mathcal D_{M}(\Qstar,\F^{\star},\delta)^{2}\,\concentration^{2} \eta_3(\Phi_{M})^{2}\,,
\end{align*}
with
\begin{align*}
\mathcal C_{M}(\Qstar,\F^{\star},\delta)\eqdef& \frac2{\sqrt M}\frac{\displaystyle\max_{x\in\X}\lVert f^{\star}_{x}\lVert_{2}}{\sigmaFamily^2\pistar_{\mathrm{min}}\sigma_{K}({\Qstar}^{2})} + \left[1+\frac{\lVert g^{\star}\lVert_{2}}{\pistar_{\mathrm{min}}\sigmaFamily^2\sigma_{K}({\Qstar}^{2})}\frac1{\sqrt M}\right]\\ 
&\hspace{-.3cm}\times \left[\frac{13\kappa^{2}(\Qstar)K^{1/2}}{\pistar_{\mathrm{min}}\sigma_{K}({\Qstar}^{2})}\frac{\kappa^{2}_{\F^{\star}}}{\sigmaFamily^2}+\frac{83}{\delta}\frac{\kappa^{6}(\Qstar)K^{5}}{\pistar_{\mathrm{min}}\sigma_{K}({\Qstar}^{2})}\frac{\kappa^{6}_{\F^{\star}}\displaystyle\max_{k\in\X}\lVert f^{\star}_{k}\lVert_{2}}{\sigmaFamily^3}\left\{1+\left({2\log\frac{K^{2}}\delta}\right)^{1/2}\right\}\right]\eqsp,\\
\mathcal D'_{M}(\Qstar,\F^{\star},\delta)\eqdef&\frac{2}{ 3 \sigmaFamily^{2}}\left[{4\sqrt K \mathcal C_{M}(\Qstar,\F^{\star},\delta)\displaystyle\max_{x\in\X}\lVert f^{\star}_{x}\lVert_{2}}+\frac{3\sqrt 3 \sigmaFamily}M\right]\,,\\
{\mathcal D}_{M}(\Qstar,\F^{\star},\delta)\eqdef&\frac{8\lVert f^{\star}_{(Y_{1},Y_{3})}\lVert_{2}}{3\sigmaFamily^{2}{\pistar}_{\mathrm{min}}^{2}}\left[\mathcal D'_{M}(\Qstar,\F^{\star},\delta)+4\sqrt{3K}\pistar_{\mathrm{min}}\mathcal C_{M}(\Qstar,\F^{\star},\delta)+\frac{5\pistar_{\mathrm{min}}}{\lVert f^{\star}_{(Y_{1},Y_{3})}\lVert_{2}\sqrt{M}}\right]
\,,\\
{\mathcal E}_{M}(\Qstar,\F^{\star},\delta)\eqdef&\frac{16\lVert f^{\star}_{(Y_{1},Y_{3})}\lVert_{2}}{\sigma_{K}^{2}(\mathbf{A}_{\Qstar})\sigmaFamily^{2}{\pistar}_{\mathrm{min}}^{2}}\left[\mathcal D'_{M}(\Qstar,\F^{\star},\delta)+4\sqrt{3K}\pistar_{\mathrm{min}}\mathcal C_{M}(\Qstar,\F^{\star},\delta)+\frac{5\pistar_{\mathrm{min}}}{\lVert f^{\star}_{(Y_{1},Y_{3})}\lVert_{2}\sqrt{M}}\right]\eqsp,
\end{align*}
where $\kappa_{\F^{\star}}$ is given in Lemma~\ref{L4}, for all $(y_{1},y_{2},y_{3})\in\Y^{3}$,
\[
g^{\star}\left(y_{1},y_{2},y_{3}\right)\eqdef\sum_{x_1,x_{2},x_3\in\X} \pistar(x_1) \Qstar(x_1,x_2)\Qstar(x_{2},x_3)\fstar_{x_1}(y_1) \fstar_{x_2}(y_2) \fstar_{x_3}(y_3)\,,\] 
and $\sigma_{K}^{2}(\mathbf{A}_{\Qstar})$ is the $K$-th largest singular value of
$
\begin{pmatrix}
\mathrm{Id}_{K}-(\Qstar)^{\top}\\
\one_{K}^{\top}
\end{pmatrix}
$ $($which is positive, see \eqref{eq:KPerron}$)$.
\end{theorem}
Theorem~\ref{thm:ThmMain1} is proved using the analysis of \cite{anandkumar2012method} to control the $\mathrm{L}^{2}$-error of the estimation based on the spectral method described in Section~\ref{sec:SpectralMethod}. To use their result in the nonparametric framework, it is essential to state explicitly how all constants depend on the dimension $M$. We thus need to recast and optimize the results of \cite{anandkumar2012method}. This is done in Theroem~\ref{thm:HSK} which is proved in Appendix~\ref{sec:HSK}. Define 
\begin{equation}
\label{eq:def:gamma}
\gamma(\O)\eqdef\displaystyle\min_{x_{1} \neq x_{2}}\left\lVert\O(\ldotp,x_{1})-\O(\ldotp,x_{2})\right\lVert_{2}
\end{equation}
and for all $A\in\R^{M\times M\times M}$ and all $B\in\R^{M\times K}$
\begin{equation}
\label{eq:def:infty:2}
\lVert A\lVert_{\infty,2} \eqdef\displaystyle\max_{\lVert v\lVert_{2}=1}\left\lVert \sum_{b=1}^{M}v_{b}A(\ldotp,b,\ldotp)\right\lVert\quad\mbox{and}\quad\lVert B\lVert_{2,\infty}\eqdef\max_{x\in\X}\left\lVert B(\ldotp,x)\right\lVert_{2}\eqsp.
\end{equation}

\begin{theorem}\label{thm:HSK}
Let $0<\delta<1$. Assume that $3\lVert\hatPj-\Pj\lVert\leq\sigma_{K}(\Pj)$ and that
\begin{align}
\label{eq:H2AHK}
8.2K^{5/2}(K-1)
\frac{\k^{2}(\Qstar\O^{\top})}
{\delta\gamma(\O)\sigma_{K}(\Pj)}
\left[
\lVert\hatM-\M\lVert_{\infty,2}+
\frac{\lVert\M\lVert_{\infty,2}\lVert\hatPj-\Pj\lVert}{\sigma_{K}(\Pj)}
\right]
&<1\,,
\\
\label{eq:H3AHK}
43.4K^{4}(K-1)\frac
{\k^{4}(\Qstar\O^{\top})}
{\delta\gamma(\O)\sigma_{K}(\Pj)}
\left[
\lVert\hatM-\M\lVert_{\infty,2}+
\frac{\lVert\M\lVert_{\infty,2}\lVert\hatPj-\Pj\lVert}{\sigma_{K}(\Pj)}
\right]
&\leq1\,,
\end{align}
then, with probability greater than $1-2\delta$, the matrix $\Uhat^{\top}\hatPj\Uhat$ is invertible, the random matrix $\Chat(1)$ is diagonalisable (see Algorithm \ref{alg:Spectral}), and there exists a permutation $\tau\in\mathcal S_{K}$ such that for all $x\in\X$, 
\begin{multline*}
\lVert\O(\ldotp,x)-\hatO(\ldotp,\tau(x))\lVert_{2} \le \frac{2\lVert\hatPj-\Pj\lVert}{\sigma_{K}(\Pj)}\lVert\O\lVert_{2,\infty} + \left[
\lVert\hatM-\M\lVert_{\infty,2}+
\frac{\lVert\M\lVert_{\infty,2}\lVert\hatPj-\Pj\lVert}{\sigma_{K}(\Pj)}
\right]\\
\times
\left[
13K^{1/2}\frac{\k^{2}(\Qstar\O^{\top})}{\sigma_{K}(\Pj)}+116K^{5}\left\{1+\left({2\log(K^{2}/\delta)}\right)^{1/2}\right\}\frac{\k^{6}(\Qstar\O^{\top})\lVert\O\lVert_{2,\infty}}{\delta\gamma(\O)\sigma_{K}(\Pj)}\right]
\eqsp.
\end{multline*}
\end{theorem}

\subsection*{Preliminary lemmas}
\label{sec:Preliminaires}
\begin{lemma}\label{L4}
There exists a constant $\kappa_{\F^{\star}}$ that depends only on $\F^{\star}$ such that for all $M\geq M_{\F^{\star}}$, $\kappa(\O)\leq\kappa_{\F^{\star}}$ where $M_{\F^{\star}}$ is given in Lemma~\ref{lem:sigmastar}.  For all $M\geq M_{\F^{\star}}$, $\kappa(\Qstar\O^{\top})\leq\kappa_{\F^{\star}}\kappa(\Qstar)$.
\end{lemma}
\begin{proof}
Note that $\Ostar^{\top}\Ostar$ is nonsingular. From  \eqref{eq:ConvergenceL2} and \eqref{eq:OMerror} we deduce that $\O^{\top}\O$ tends to $\Ostar^{\top}\Ostar$ as $M$ grows to infinity. This proves the first point. Recall that $\sigma_{i}(AB)\leq\sigma_{1}(A)\sigma_{i}(B)$ for all $i=1,\dots, K$. Applying this identity to $A={\Qstar}^{-1}$ and $B=\Qstar\O^{\top}$ yields $\sigma_{K}(\Qstar)\sigma_{K}(\O)\leq\sigma_{K}(\Qstar\O^{\top})$. It follows that $\kappa(\Qstar\O^{\top})\leq\kappa(\Qstar)\kappa(\O)$. The second claim follows from the first claim.
\end{proof}

\begin{lemma}
For all $M\geq M_{\F^{\star}}$, $\gamma(\O)\geq\sqrt 2\sigmaFamily$ and $\lVert \O\lVert_{2,\infty}\leq\max_{x\in\X}\lVert f^{\star}_{x}\lVert_{2}$, where $\gamma(\O)$ and $\lVert \O\lVert_{2,\infty}$ are defined in \eqref{eq:def:gamma} and \eqref{eq:def:infty:2}.
\end{lemma}
\begin{proof}
Observe that $\lVert\O v\lVert_{2}\geq\sigma_{K}(\O)\lVert v\lVert_{2}$. With an appropriate choice of $v$ and using Lemma~\ref{lem:sigmastar} this proves the first inequality. As $\Phi_{M}$ is an orthonormal family,  $\lVert \O(\ldotp,x)\lVert_{2}\leq\lVert f^{\star}_{x}\lVert_{2}$ which proves the second claim.
\end{proof}

\begin{lemma}
\label{lem:BorneM_INFTY2}
For all $M\ge 1$, 
\[
\lVert \M\lVert_{\infty,2}\eqdef\displaystyle\max_{\lVert v\lVert_{2}=1}\left\lVert \sum_{b=1}^{M}v_{b}\M(\ldotp,b,\ldotp)\right\lVert
\leq\lVert g^{\star}\lVert_{2}\eqsp,
\] 
where $\Vert\cdot\Vert_{\infty,2}$ is defined in \eqref{eq:def:infty:2}.
\end{lemma}
\begin{proof}
As for all $x\in\X$, $\fstar_x\in\mathrm{L}^{2}(\Y,\L)$, $g^{\star}\in\mathrm L^{2}(\Y^{3},{\L}^{\otimes 3})$. Denote by $\langle\ldotp,\ldotp\rangle_{\mathrm L^{2}(\Y^{3},{\L}^{\otimes 3})}$ the inner product of $\mathrm L^{2}(\Y^{3},{\L}^{\otimes 3})$. As $\varphi_{a,b,c}(y_{1},y_{2},y_{3})\eqdef\varphi_{a}(y_{1})\varphi_{b}(y_{2})\varphi_{c}(y_{3})$ is an orthonormal family of $\mathrm L^{2}(\Y^{3},{\L}^{\otimes 3})$,
\begin{align*}
\lVert \M\lVert_{\infty,2}&=\max_{\lVert v\lVert_{2}=1}\left\lVert \sum_{b=1}^{M}v_{b}\M(\ldotp,b,\ldotp)\right\lVert\leq\max_{\lVert v\lVert_{2}=1} \sum_{b=1}^{M}\lvert v_{b}\lvert\lVert\M(\ldotp,b,\ldotp)\lVert\,,\\
&\leq\left(\sum_{b=1}^{M}\lVert\M(\ldotp,b,\ldotp)\lVert^{2}\right)^{1/2}\leq\left(\sum_{b=1}^{M}\lVert\M(\ldotp,b,\ldotp)\lVert_{F}^{2}\right)^{1/2}\,,\\
&=\left(\sum_{a,b,c=1}^{M}\E\left[\varphi_{a}(Y_{1})\varphi_{b}(Y_{2})\varphi_{c}(Y_{3})\right]^{2}\right)^{1/2}=\left(\sum_{a,b,c=1}^{M}\langle g^{\star},\varphi_{a,b,c}\rangle_{\mathrm L^{2}(\Y^{3},{\L}^{\otimes 3})}^{2}\right)^{1/2}\leq\lVert g^{\star}\lVert_{2}\,.
\end{align*}
using Cauchy-Schwarz inequality.
\end{proof}

\begin{lemma}
For all $M\ge 1$, $\lVert \hatM-\M\lVert_{\infty,2}\leq\lVert \hatM-\M\lVert_{F}$, where $\Vert\cdot\Vert_{\infty,2}$ is defined in \eqref{eq:def:infty:2}.
\end{lemma}
\begin{proof}
For all $M\ge 1$,
\begin{align*}
\lVert \hatM-\M\lVert_{\infty,2}&=\max_{\lVert v\lVert_{2}=1}\left\lVert \sum_{b=1}^{M}v_{b}(\hatM-\M)(\ldotp,b,\ldotp)\right\lVert\leq\max_{\lVert v\lVert_{2}=1} \sum_{b=1}^{M}\lvert v_{b}\lvert\left\lVert(\hatM-\M)(\ldotp,b,\ldotp)\right\lVert\,,\\
&\leq\left(\sum_{b=1}^{M}\left\lVert(\hatM-\M)(\ldotp,b,\ldotp)\right\lVert^{2}\right)^{1/2}\leq\left(\sum_{b=1}^{M}\left\lVert(\hatM-\M)(\ldotp,b,\ldotp)\right\lVert_{F}^{2}\right)^{1/2}\,,\\
&=\left\lVert \hatM-\M\right\lVert_{F}\,.
\end{align*}
using Cauchy-Schwarz inequality.
\end{proof}

\begin{lemma}
\label{lem:sigmaKP}
Under {\bf [H\ref{assum:Qstarprime}']} and {\bf [H\ref{assum:stationary}]}, for all $M\ge 1$, $\sigma_{K}(\Pj)\geq\pi_{\mathrm{min}}\sigma_{K}^{2}(\O)\sigma_{K}(\Q^{2})$. If {\bf [H\ref{assum:linearly:ind}]}  holds, then, for all $M\geq M_{\F^{\star}}$, 
\[
\sigma_{K}(\Pj)\geq{\sigmaFamily^2\pistar_{\mathrm{min}}\sigma_{K}({\Qstar}^{2})}\eqsp,
\]
where $M_{\F^{\star}}$ and $\sigmaFamily$ are defined in Lemma~\ref{lem:sigmastar}.
\end{lemma}
\begin{proof}
By Lemma \ref{lem:ExpressionMetP} and \eqref{eq:pistarmin}, 
\begin{align*}
\sigma_{K}(\Pj) =\sigma_{K}(\U^{\top}\Pj\U) &=\sigma_{K}((\U^{\top}\O)\D{\pistar}{\Qstar}^{2}(\U^{\top}\O)^{\top})\eqsp,\\
&\geq{\sigma_{K}(\U^{\top}\O)}{\sigma_{K}(\D{\pistar}{\Qstar}^{2}(\U^{\top}\O)^{\top})}\eqsp,\\
&=\sigma_{K}(\O){\sigma_{K}(\D{\pistar}{\Qstar}^{2}(\U^{\top}\O)^{\top})}\eqsp,\\
&\geq\sigma_{K}(\D{\pistar})\sigma_{K}(\O)\sigma_{K}((\U^{\top}\O)^{\top})\sigma_{K}({\Qstar}^{2})\eqsp,\\
&=\pistar_{\mathrm{min}}\sigma_{K}^{2}(\O)\sigma_{K}({\Qstar}^{2})\eqsp,
\end{align*}
which concludes the proof.
\end{proof}

\subsection*{First step: Estimation of the emission laws using a spectral method} 
Appendix \ref{sec:AppendixConcentration} shows that:
\begin{eqnarray*}
\P\Big[\lVert\hatl-\l\lVert_{F}\geq\concentration\eta_1(\Phi_{M})/\sqrt{p}\Big]\leq\delta'\eqsp, &  \P\Big[\lVert\hatM-\M\lVert_{F}\geq\concentration\eta_3(\Phi_{M})/\sqrt{p}\Big]\leq\delta'\eqsp,\\
\P\Big[\lVert\hatN-\N\lVert_{F}\geq\concentration \eta_2(\Phi_{M})/\sqrt{p}\Big]\leq\delta'\eqsp, & \P\Big[\lVert\hatPj-\Pj\lVert_{F}\geq\concentration \eta_2(\Phi_{M})/\sqrt{p}\Big]\leq\delta'\eqsp.
\end{eqnarray*}
Using the preliminary lemmas of Section \ref{sec:Preliminaires} and the elementary fact that $M\eta_{1}(\Phi_{M})\leq\sqrt M\eta_2(\Phi_{M})\leq\eta_3(\Phi_{M})$, deduce that \eqref{eq:H2AHK} and \eqref{eq:H3AHK} along with $3\lVert\hatPj-\Pj\lVert\leq\sigma_{K}(\Pj)$ are satisfied when $M\geq M_{\F^{\star}}$ and $p\geq \borneN$ where:
\[
\borneN\eqdef\frac{942}{\delta^{2}}\frac{\kappa^{8}(\Qstar)K^{10}}{{\pistar}_{\mathrm{min}}^{2}\sigma_{K}^{2}({\Qstar}^{2})}\frac{\kappa^{8}_{\F^{\star}}}{\sigmaFamily^6}\Big(1+\frac{\lVert g^{\star}\lVert_{2}}{\pistar_{\mathrm{min}}\sigmaFamily^2\sigma_{K}({\Qstar}^{2})}\frac1{\sqrt M}\Big)^{2}\concentration^{2} \eta_3(\Phi_{M})^{2}\,.
\]
Using Theorem \ref{thm:HSK}, with probability greater than $1-2\delta-4\delta'$, there exists a permutation $\tau$ satisfying for any $M\geq M_{\F^{\star}}$, $p\geq \borneN$ and $x\in\X$,
\eq
\notag
\lVert\O(\ldotp,x)-\hatO(\ldotp,\tau(x))\lVert_{2}\leq\mathcal {C}_{M}(\Qstar,\F^{\star},\delta)\concentration\eta_3(\Phi_{M})/\sqrt p\eqsp.
\qe
\noindent
This proves the first part of Theorem~\ref{thm:ThmMain1}.

\subsection*{Second step: Preliminary estimation of the stationary density using a spectral method} 
For sake of readability, assume that $\tau$ is the identity permutation. Observe that: 
\[
\borneNfinale\geq\borneN\,.
\]
Recall $\tilde\pi\eqdef\big(\hat\U^{\top}\hatO\big)^{-1}\hat\U^{\top}\hatl$
and $\pistar=\big(\hat\U^{\top}\O\big)^{-1}\hat\U^{\top}\l$. 
\begin{lemma}
\label{lem:estimationdePi} 
With probability greater than $1-2\delta-4\delta'$, if $p>\borneNfinale$ then,
\begin{multline*}
\lVert\tilde\pi-\pistar\lVert_{2}
\leq
\frac{2}{\sqrt 3 \sigmaFamily}\concentration\frac{ \eta_{1}(\Phi_{M})}{\sqrt p}\\ 
+\frac{2}{\sqrt 3 \sigmaFamily}\frac{\sqrt{\borneNfinale}}{\sqrt p-\sqrt{\borneNfinale}}\left(\max_{x\in\X}\lVert f^{\star}_{x}\lVert_{2}\,+\,\concentration\frac{ \eta_{1}(\Phi_{M})}{\sqrt p}\right)\,.
\end{multline*}
\end{lemma}
\begin{proof}
Set $A=\hat\U^{\top}\O$, $\tilde A=\hat\U^{\top}\hatO$ and $B=\hat\U^{\top}(\O-\hatO)$. Then,
\[
\lVert B\lVert\leq\lVert\O-\hatO\lVert\leq\lVert\O-\hatO\lVert_{F}\leq\sqrt K\max_{x\in\X}\lVert\O(\ldotp,x)-\hatO(\ldotp,x)\lVert_{2}\,,
\]
which gives $\lVert B\lVert\leq\sqrt K\mathcal {C}_{M}(\Qstar,\F^{\star},\delta)\concentration{\eta_{3}(\Phi_{M})}/{\sqrt p}$. Similarly, by claim (iii) of Lemma~\ref{lem:Wedinsetal}:
\[
\lVert A^{-1}B\lVert\leq\lVert A^{-1}\lVert\lVert B\lVert\leq\sigma_{K}^{-1}(A)\lVert B\lVert\leq\frac{2\sqrt K\max_{x\in\X}\lVert\O(\ldotp,x)-\hatO(\ldotp,x)\lVert_{2}}{\sqrt 3\sigma_{K}(\O)}\,,
\]
so that
\[
\lVert A^{-1}B\lVert\leq\frac{2\sqrt K}{\sqrt 3\sigmaFamily}\mathcal {C}_{M}(\Qstar,\F^{\star},\delta)\concentration\frac{\eta_{3}(\Phi_{M})}{\sqrt p}\,.
\]
Observe that the condition on $p$ and $M$ ensures that $\lVert A^{-1}B\lVert<1$. Apply Theorem \ref{thm:InversePerturbation} to get that:
\eq
\label{eq:ControlInvers1}
\lVert(\hat\U^{\top}\O)^{-1}-(\hat\U^{\top}\hatO)^{-1}\lVert\leq\frac{2}{\sqrt 3 \sigmaFamily}\frac{\sqrt{\borneNfinale}}{\sqrt p-\sqrt{\borneNfinale}}\,.
\qe
Furthermore, using \eqref{eq:ControlInvers1}:
\begin{align*}
\lVert\tilde\pi-\pistar\lVert_{2}&=\lVert\big(\hat\U^{\top}\hatO\big)^{-1}\hat\U^{\top}\hatl-\big(\hat\U^{\top}\O\big)^{-1}\hat\U^{\top}\l\lVert_{2}\\
&=\lVert\big(\hat\U^{\top}\hatO\big)^{-1}\hat\U^{\top}\hatl-\big(\hat\U^{\top}\O\big)^{-1}\hat\U^{\top}\hatl+\big(\hat\U^{\top}\O\big)^{-1}\hat\U^{\top}\hatl-\big(\hat\U^{\top}\O\big)^{-1}\hat\U^{\top}\l\lVert_{2}\\
&\leq\lVert(\hat\U^{\top}\O)^{-1}-(\hat\U^{\top}\hatO)^{-1}\lVert\lVert\hatl\lVert_{2}+\lVert A^{-1}\lVert\lVert\hatl-\l\lVert_{2}\\
&\leq\frac{2}{\sqrt 3 \sigmaFamily}\left(\lVert\hatl-\l\lVert_{2}+\frac{\sqrt{\borneNfinale}}{\sqrt p-\sqrt{\borneNfinale}}(\lVert\l\lVert_{2}+\lVert\hatl-\l\lVert_{2})\right)\,.
\end{align*}
Denote $f^{\star}_{Y_{1}}=\sum_{x_1\in\X} \pi(x_1) f^{\star}_{k_1}(y_1)$ the density of $Y_1$. Observe that:
\begin{align*}
\lVert\l\lVert_{2}&=\left(\sum_{a=1}^{M}\E\left[\varphi_{a}(Y_{1})\right]^{2}\right)^{1/2}=\left(\sum_{a=1}^{M}\langle f^{\star}_{Y_{1}},\varphi_{a}\rangle^{2}\right)^{1/2}\leq\lVert f^{\star}_{Y_{1}}\lVert_{2}\leq\max_{x\in\X}\lVert f^{\star}_{x}\lVert_{2}\eqsp,
\end{align*}
which concludes the proof.
\end{proof}
\noindent This results allows to state that for all $p\geq4\borneNfinale$,
\eq
\label{eq:EstimationTildeL}
\lVert\pistar-\mathbb P_{\tau}\tilde\pi\lVert_{2}\leq\mathcal {D}'_{M}(\Qstar,\F^{\star},\delta)\concentration\eta_3(\Phi_{M})/\sqrt p\eqsp.
\qe

\subsection*{Third step: Estimation of the transition matrix using a spectral method} 
Denote $\tilde\Q\eqdef\big(\hat\U^{\top}\hatO\D{\tilde\pi}\big)^{-1}\hat\U^{\top}\hatN\hat\U\big(\hatO^{\top}\hat\U\big)^{-1}$. Observe $\hat\Q=\Pi_{TM}(\tilde\Q)$ and $\Qstar=\Pi_{TM}(\Qstar)$ and hence, by non-expansivity of the projection onto convex sets, $\lVert\hat\Q-\Qstar\lVert_{F}\leq\lVert\tilde\Q-\Qstar\lVert_{F}$. Moreover, notice that: 
\[
\borneNfinalebis\geq4\borneNfinale\geq\borneN\,.
\]

\begin{lemma}
\label{lem:EstimationQ}  
With probability greater than $1-2\delta-4\delta'$, if $p\geq\borneNfinalebis$ then
\[
\lVert\tilde\Q-\Qstar\lVert\leq \frac{8\lVert f^{\star}_{(Y_{1},Y_{3})}\lVert_{2}}{3\sigmaFamily^{2}{\pistar}_{\mathrm{min}}^{2}}\lVert\tilde\pi-\pistar\lVert_{2}+ \frac{2}{\pistar_{\mathrm{min}}}\tilde{\mathcal E}_{M}(\Qstar,\F^{\star},\delta)\concentration\frac{\eta_{3}(\Phi_{M})}{\sqrt p}\,,
\]
where
\[
\tilde{\mathcal E}_{M}(\Qstar,\F^{\star},\delta)\eqdef\frac{16}{{\sqrt3 \sigmaFamily^{2}}}\bigg[{\sqrt K}\mathcal C_{M}(\Qstar,\F^{\star},\delta)\lVert f^{\star}_{(Y_{1},Y_{3})}\lVert_{2}+\frac{5}{4\sqrt{3M}}\bigg]\,.
\]
\end{lemma}
\begin{proof}
Observe that \eqref{eq:EstimationOfL} shows that $\lVert\tilde\pi-\pistar\lVert_{2}\leq\pistar_{\mathrm{min}}/2$. Then, for any $x\in\X$: 
\eq
\label{eq:bornehatpimin}
\tilde\pi_{x}\geq\frac{\pistar_{\mathrm{min}}}2>0\,.
\qe
Set $\V=(\hat\U^{\top}\O)^{-1}\hat\U^{\top}$ and $\hat\V=(\hat\U^{\top}\hatO)^{-1}\hat\U^{\top}$. Note $\tilde\Q=\D{\tilde\pi}^{-1}\hat\V\hatN\hat\V^{\top}$ and: 
\[
\Q=\D{\pistar}^{-1}\V\N\V^{\top}\,.
\] 
Set $E=\hat\V-\V$ and $F=\hatN-\N$. Using \eqref{eq:ControlInvers1} yields:
\[
\lVert E\lVert\leq\frac{2}{\sqrt 3 \sigmaFamily}\frac{\sqrt{\borneNfinale}}{\sqrt p-\sqrt{\borneNfinale}}\leq\frac{8\sqrt K}{3 \sigmaFamily^{2}}\mathcal C_{M}(\Qstar,\F^{\star},\delta)\concentration\frac{\eta_{3}(\Phi_{M})}{\sqrt p}\,.
\]
By claim (iii) of Lemma~\ref{lem:Wedinsetal}, $\lVert \V\lVert\leq\sigma_{K}^{-1}(\hat\U^{\top}\O)\leq 2/(\sqrt 3\sigmaFamily)$. Furthermore, $\varphi_{a,c}(y_{1},y_{3})\eqdef\varphi_{a}(y_{1})\varphi_{c}(y_{3})$ is an orthonormal family of $\mathrm L^{2}(\Y^{2},{\L}^{\otimes 2})$ and
\[
\lVert\N\lVert_{F}=\Big(\sum_{a,c=1}^{M}\E\left[\varphi_{a}(Y_{1})\varphi_{c}(Y_{3})\right]^{2}\Big)^{1/2}=\Big(\sum_{a,c=1}^{M}\langle f^{\star}_{(Y_{1},Y_{3})},\varphi_{a,c}\rangle_{\mathrm L^{2}(\Y^{2},{\L}^{\otimes 2})}^{2}\Big)^{1/2}\leq\lVert f^{\star}_{(Y_{1},Y_{3})}\lVert_{2}\,.
\]
Then,
\begin{align*}
\lVert\V\N\V^{\top}-\hat\V\hatN\hat\V^{\top}\lVert&=\lVert\V\N\V^{\top}-(\V+E)(\N+F)(\V+E)^{\top}\lVert\,,\\
&=\lVert\V\N E^{\top}+\V F\V^{\top}+\V FE^{\top}+E\N\V^{\top}+E\N E^{\top}+EF\V^{\top}+EFE^{\top}\lVert\,,\\
&\leq 2\lVert E\lVert\lVert\V\lVert\lVert\N\lVert+2\lVert E\lVert\lVert\V\lVert\lVert F\lVert+\lVert E\lVert^{2}\lVert \N\lVert+\lVert \V\lVert^{2}\lVert F\lVert+\lVert E\lVert^{2}\lVert F\lVert\,,
\end{align*}
yields
\begin{align*}
\notag
\lVert\V\N\V^{\top}-\hat\V\hatN\hat\V^{\top}\lVert\leq&\frac{32\sqrt K\mathcal C_{M}(\Qstar,\F^{\star},\delta)\concentration\lVert f^{\star}_{(Y_{1},Y_{3})}\lVert_{2}}{3\sqrt 3\sigmaFamily^{3}}\bigg[1+\frac{\concentration}{\lVert f^{\star}_{(Y_{1},Y_{3})}\lVert_{2}}\frac{\eta_{3}(\Phi_{M})}{\sqrt{pM}}
\\
\notag
&+\frac{2\sqrt K\mathcal C_{M}(\Qstar,\F^{\star},\delta)\concentration}{\sqrt 3\sigmaFamily}\frac{\eta_{3}(\Phi_{M})}{\sqrt p}
\\
\notag
&
+\frac{\sqrt 3\sigmaFamily}{4\mathcal C_{M}(\Qstar,\F^{\star},\delta)\lVert f^{\star}_{(Y_{1},Y_{3})}\lVert_{2}\sqrt{K}}\frac{1}{\sqrt M}
\\
\notag
&
+\frac{2\sqrt K\mathcal C_{M}(\Qstar,\F^{\star},\delta)\concentration^{2}}{\sqrt 3\sigmaFamily\lVert f^{\star}_{(Y_{1},Y_{3})}\lVert_{2}}\frac{\eta^{2}_{3}(\Phi_{M})}{p\sqrt M}
\bigg]\frac{\eta_{3}(\Phi_{M})}{\sqrt p}
\end{align*}
As $p\geq\borneNfinalebis\geq4\borneNfinale=\frac{16K}{3\sigmaFamily^2}\mathcal C_{M}(\Qstar,\F^{\star},\delta)^{2}\,\concentration^{2} \eta_3(\Phi_{M})^{2}$,
\eq
\label{eq:ControlQD}
\lVert\V\N\V^{\top}-\hat\V\hatN\hat\V^{\top}\lVert\leq\tilde{\mathcal E}_{M}(\Qstar,\F^{\star},\delta)\concentration\frac{\eta_{3}(\Phi_{M})}{\sqrt p}\eqsp.
\qe
Observe that:
\begin{align*}
\lVert\Qstar-\tilde\Q\lVert&=\lVert(\D{\pistar}^{-1}-\D{\pihat}^{-1})\V\N\V^{\top}+\D{\pihat}^{-1}(\V\N\V^{\top}-\hat\V\hatN\hat\V^{\top})\lVert\\
&\leq\lVert\D{\pistar}^{-1}-\D{\pihat}^{-1}\lVert\lVert\V\lVert^{2}\lVert\N\lVert+\lVert\D{\pihat}^{-1}\lVert\lVert\V\N\V^{\top}-\hat\V\hatN\hat\V^{\top}\lVert\\
&\leq \frac{4\lVert f^{\star}_{(Y_{1},Y_{3})}\lVert_{2}}{3\sigmaFamily^{2}}\max_{x\in\X}({\pistar}_{x}^{-1}-\tilde\pi_{x}^{-1})+\max_{x\in\X}\pihat_{x}^{-1}\tilde{\mathcal E}_{M}(\Qstar,\F,\delta)\concentration\frac{\eta_{3}(\Phi_{M})}{\sqrt p}\\
&\leq \frac{8\lVert f^{\star}_{(Y_{1},Y_{3})}\lVert_{2}}{3\sigmaFamily^{2}{\pistar}_{\mathrm{min}}^{2}}\lVert\tilde\pi-\pistar\lVert_{2}+ \frac{2}{\pistar_{\mathrm{min}}}\tilde{\mathcal E}_{M}(\Qstar,\F^{\star},\delta)\concentration\frac{\eta_{3}(\Phi_{M})}{\sqrt p}\,,
\end{align*}
using \eqref{eq:bornehatpimin} and \eqref{eq:ControlQD}. 
\end{proof}
\noindent
Combining \eqref{eq:EstimationTildeL} and Lemma \ref{lem:EstimationQ} proves the second point of Theorem \ref{thm:ThmMain1}.

\subsection*{Last step: Final estimation of the stationary distribution} 
By {\bf [H\ref{assum:Qstar}']}, we know that the transition matrix $\Qstar$ is irreducible and aperiodic. Perron-Frobenius theorem shows that $\Qstar$ has a unique stationary distribution $\pistar$. More precisely, 
\begin{enumerate}[-]
\item $\mathbb R\,.\,\pistar=\ker(\mathrm{Id}_{K}-(\Qstar)^{\top})$ so that $(\mathbb R\,.\,\pistar)^{\bot}=\mathrm{range}(\mathrm{Id}_{K}-\Qstar)$,
\item and $\langle \pistar,\one_{K}\rangle=1$,
\end{enumerate}
where $\one_{K}=(1,\ldots,1)\in\R^{K}$. We deduce $\one_{K}\notin\mathrm{range}(\mathrm{Id}_{K}-\Qstar)$ and 
\eq
\label{eq:KPerron}
\mathrm{Rank}
\begin{pmatrix}
\mathrm{Id}_{K}-(\Qstar)^{\top}\\ \one_{K}^{\top}
\end{pmatrix}=K\,.
\qe
Set 
\[
A=\begin{pmatrix}
\mathrm{Id}_{K}-\Q^{\top}\\ \one_{K}^{\top}
\end{pmatrix}\quad\mathrm{and}\quad A^{\star}=\begin{pmatrix}
\mathrm{Id}_{K}-(\Qstar)^{\top}\\ \one_{K}^{\top}
\end{pmatrix}\,.
\]
We first derive an upper bound on $\lVert A^{+}-(A^{\star})^{+}\lVert$ where $A^{+}$ denotes the Moore-Penrose pseudo-inverse of $A$. Note that
\eq
\label{eq:interPIHAT}
A^{+}-(A^{\star})^{+}=(A^{\star})^{+}(A^{\star}-A)A^{+}-(A^{\star})^{+}(\mathrm{Id}_{K+1}-AA^{+})\,.
\qe
The last term can be written as 
\[
(A^{\star})^{+}(\mathrm{Id}_{K+1}-AA^{+})=(A^{\star})^{+}(A^{\star}(A^{\star})^{+})(\mathrm{Id}_{K+1}-AA^{+})=(A^{\star})^{+}P_{\mathrm{range}(A^{\star})}P_{\mathrm{range}(A)^{\bot}}\eqsp,
\]
where $P_{\mathrm{range}(A^{\star})}=A^{\star}(A^{\star})^{+}$ denotes the orthogonal projection onto $\mathrm{range}(A^{\star})$ and $P_{\mathrm{range}(A)^{\bot}}=\mathrm{Id}_{K+1}-AA^{+}$ denotes the orthogonal projection onto the orthogonal of $\mathrm{range}(A)$. Define 
\begin{equation}
\label{eq:def:s}
s(\Qstar)\eqdef\sigma_{K}(A^{\star})\eqsp.
\end{equation}
\begin{lemma}
If $\lVert \Q-\Qstar\lVert\leq s(\Qstar)/2$ then $\mathrm{Rank}(A)=\mathrm{Rank}(A^{\star})=K$ and
\[
\lVert P_{\mathrm{range}(A^{\star})}P_{\mathrm{range}(A)^{\bot}}\lVert\leq\frac{2\lVert \Q-\Qstar\lVert}{s(\Qstar)}\,.
\]
\end{lemma}

\begin{proof}
The first point follows from Weyl's inequality, see Theorem \ref{thm:Weyl}. By \cite{wedin1972perturbation}, 
\[
\lVert P_{\mathrm{range}(A^{\star})^{\bot}}P_{\mathrm{range}(A)}\lVert=\lVert P_{\mathrm{range}(A)^{\bot}}P_{\mathrm{range}(A^{\star})}\lVert\,.
\]
Moreover, since projections $P$ are orthogonal $(P_{\mathrm{range}(A)^{\bot}}P_{\mathrm{range}(A^{\star})})^{\top}=P_{\mathrm{range}(A^{\star})}P_{\mathrm{range}(A)^{\bot}}$. Using notation of \cite{wedin1972perturbation}, one may notice that $\lVert\sin\theta(\mathrm{range}(A),\mathrm{range}(A^{\star}))\lVert=\lVert P_{\mathrm{range}(A^{\star})^{\bot}}P_{\mathrm{range}(A)}\lVert$. By Wedin's theorem \cite{wedin1972perturbation}, if $\sigma_{K}(A)\geq s(\Qstar)/2
$ then $\lVert\sin\theta(\mathrm{range}(A),\mathrm{range}(A^{\star}))\lVert\leq\frac{2\lVert A-A^{\star}\lVert}{\sigma_{K}(A^{\star})}$.
We conclude using Weyl's inequality, see Theorem \ref{thm:Weyl}.
\end{proof}
\noindent
Triangular inequality in \eqref{eq:interPIHAT} gives
\begin{align*}
\lVert A^{+}-(A^{\star})^{+} \lVert &\leq \lVert (A^{\star})^{+}\lVert \lVert \Q-\Qstar\lVert\Big(\lVert A^{+}\lVert+\frac{2}{\sigma_{K}(A^{\star})}\Big)\,,\\
&\leq \frac{\lVert \Q-\Qstar\lVert}{\sigma_{K}(A^{\star})}\Big(\lVert A^{+}-(A^{\star})^{+} \lVert+\frac{3}{\sigma_{K}(A^{\star})}\Big)\,,
\end{align*}
using that $\lVert (A^{\star})^{+}\lVert=1/\sigma_{K}(A^{\star})$. Deduce that if $\lVert \Q-\Qstar\lVert\leq \sigma_{K}(A^{\star})/2$ then $\lVert A^{+}-(A^{\star})^{+} \lVert\leq{6\lVert \Q-\Qstar\lVert}/{\sigma^{2}_{K}(A^{\star})}$. From Weyl's inequality, if $\lVert \Q-\Qstar\lVert\leq \sigma_{K}(A^{\star})/2$ then $\sigma_{K}(A)\geq\sigma_{K}(A^{\star})/2$.  $\mathrm{Id}_{K}-\Q^{\top}$ has rank $K-1$ and the eigenspace $\ker(\mathrm{Id}_{K}-\Q^{\top})$ has dimension $1$.
Thus, $\Q$ is an irreducible and aperiodic transition matrix, and $\pi$ is the unique solution to
\[
\begin{pmatrix}
\mathrm{Id}_{K}-\Q^{\top}\\ \one_{K}^{\top}
\end{pmatrix}
\pi=
\begin{pmatrix}
0\\ 1
\end{pmatrix}\,.
\]
Now $\lVert \pi-\pistar\lVert_{2}\leq\lVert A^{+}-(A^{\star})^{+}\lVert$ and the last part of Theorem \ref{thm:ThmMain1} is proved.

\section{Matrix perturbation}
We gather in this section some useful results in matrix perturbation theory. Proofs of the following theorem may be found in \cite{stewart1990matrix} for instance.
\begin{theorem}[Weyl's inequality]
\label{thm:Weyl}
Let $A,B$ be $(p\times q)$ matrices with $p\geq q$ then, for all $ i=1,\ldots,q$,
\[
\lvert\sigma_{i}(A+B)-\sigma_{i}(A)\lvert\leq\sigma_{1}(B)\,.
\]
\end{theorem}

\begin{theorem}
\label{thm:InversePerturbation}
Let $A,B$ be $(p\times p)$ matrices. If $A$ is invertible and $\lVert A^{-1}B\lVert<1$ then $\tilde A\eqdef A+B$ is invertible and
\[
\lVert \tilde A^{-1}-A^{-1}\lVert\leq\frac{\lVert B\lVert\lVert A^{-1}\lVert^{2}}{1-\lVert A^{-1}B\lVert}\,.
\]
\end{theorem}

\begin{theorem}[Bauer-Fike]
\label{thm:BauerFike}
Let $A,B$ be $(p\times p)$ matrices and $\tilde A\eqdef A+B$. Assume that $A$ is diagonalizable, i.e. $X^{-1}AX=\Lambda$, where $\Lambda=\D{(\lambda_{1},\ldots,\lambda_{p})}$. Then,
\eq
\label{eq:SpectralVariationBF}
\sv_{A}(\tilde A)\leq\kappa(X)\lVert B\lVert\,,
\qe
where $\displaystyle\sv_{A}(\tilde A)\eqdef\max_{j}\min_{i}\lvert\tilde\lambda_j-\lambda_{i}\lvert$ and $\tilde\lambda_j$ denotes the eigenvalues of $\tilde A$.
\end{theorem}
\begin{remark}
Moreover, if the disks $\mathcal D_{i}\eqdef\{\xi\ :\ \lvert\xi-\lambda_{i}\lvert\leq\kappa(X)\lVert B\lVert\}$ are isolated from the others, then \eqref{eq:SpectralVariationBF} holds with the matching distance $\md(A,\tilde A)\leq\kappa(X)\lVert B\lVert$ where $\displaystyle\md(A,\tilde A)\eqdef\min_{\tau\in\mathcal S_{p}}\max_{i}\lvert\hat\lambda_{\tau(i)}-\lambda_{i}\lvert$. Eventually, if $\Lambda$, $\tilde A$ are real valued matrices then $\tilde A$ has $p$ distinct real eigenvalues.
\end{remark}

\section{Concentration inequalities}
\label{sec:AppendixConcentration}
Consider consecutive observations of the same hidden Markov chain $Z_{s}\eqdef(Y_{s},Y_{s+1},Y_{s+2})$ for $1\leq s\leq p$, 
\begin{lemma}
For any positive $u$, any $M$ and any $p$:
\begin{align*}
\label{eq:ConcentrationLcasDependent}
\P\Big[\lVert\hatl-\l\lVert_{F}&\geq\frac{\sqrt 2\eta_{1}(\Phi_{M})}{\sqrt{p\psg}}(1+2u\sqrt{1+\log(8/{\pistar}_{\mathrm{min}})})\Big]\leq\exp(-u^{2})\,,\\
\P\Big[\lVert\hatM-\M\lVert_{F}&\geq\frac{\sqrt 2\eta_{3}(\Phi_{M})}{\sqrt{p\psg}}(1+2u\sqrt{1+\log(8/{\pistar}_{\mathrm{min}})})\Big]\leq\exp(-u^{2})\,,\\
\P\Big[\lVert\hatN-\N\lVert_{F}&\geq\frac{\sqrt 2\eta_{2}(\Phi_{M})}{\sqrt{p\psg}}(1+2u\sqrt{1+\log(8/{\pistar}_{\mathrm{min}})})\Big]\leq\exp(-u^{2})\,,\\
\P\Big[\lVert\hatPj-\Pj\lVert_{F}&\geq\frac{\sqrt 2\eta_{2}(\Phi_{M})}{\sqrt{p\psg}}(1+2u\sqrt{1+\log(8/{\pistar}_{\mathrm{min}})})\Big]\leq\exp(-u^{2})\,.
\end{align*}

\end{lemma}
\begin{proof}
Set $\zeta_{\l}(Z_{1},\ldots,Z_{p})\eqdef\lVert\hatl(Z_{1},\ldots,Z_{p})-\l\lVert_{2}$, $\zeta_{\M}(Z_{1},\ldots,Z_{p})\eqdef\lVert\hatM(Z_{1},\ldots,Z_{p})-\M\lVert_{F}$, $\zeta_{\N}(Z_{1},\ldots,Z_{p})\eqdef\lVert\hatN(Z_{1},\ldots,Z_{p})-\N\lVert_{F}$ and $\zeta_{\Pj}(Z_{1},\ldots,Z_{p})\eqdef\lVert\hatPj(Z_{1},\ldots,Z_{p})-\Pj\lVert_{F}$ where, for instance $\hatl(Z_{1},\ldots,Z_{p})$ denotes the dependence of $\hatl$ in $Z_{1},\ldots,Z_{p}$. We begin with $\zeta_{\M}$, other cases are similar. Form the difference with respect to the coordinate $i$:
\[
c_{i}\eqdef\sup_{z_{j}\in\Y^{3},z'_{i}\in\Y^{3}}\left\lvert\zeta_{\M}(z_{1},\ldots,z_{i-1},z_{i},z_{i+1},\ldots,z_{p})-\zeta_{\M}(z_{1},\ldots,z_{i-1},z'_{i},z_{i+1},\ldots,z_{p})\right\lvert\,.
\]
By the triangular inequality,
\[
c_{i}\leq\sup_{z_{j}\in\Y^{3},z'_{i}\in\Y^{3}}\left\lVert\hatM(z_{1},\ldots,z_{i-1},z_{i},z_{i+1},\ldots,z_{p})-\hatM(z_{1},\ldots,z_{i-1},z'_{i},z_{i+1},\ldots,z_{p})\right\lVert_{F}\eqsp,
\]
so that
\[
c_{i}\leq\frac1p\sup_{z_{i}\in\Y^{3},z'_{i}\in\Y^{3}}
\left(\sum_{a,b,c}\left(\varphi_{a}(y_{1}^{(i)})\varphi_{b}(y_{2}^{(i)})\varphi_{c}(y_{3}^{(i)})-\varphi_{a}({y'}_{1}^{(i)})\varphi_{b}({y'}_{2}^{(i)})\varphi_{c}({y'}_{3}^{(i)})\right)^{2}
\right)^{1/2}\eqsp.
\]
Eventually, we get that $c_{i}\leq{\eta_{3}(\Phi_{M})}/p$.  By McDiarmid's inequality \cite{paulin2012concentration}, for all $u>0$,
\[
\P(\lVert\hatM-\M\lVert_{F}\geq\E\left[\lVert\hatM-\M\lVert_{F}\right] + u)\leq\exp\left(-\frac{pu^{2}}{8\mixtime\eta^{2}_{3}(\Phi_{M})}\right)\,.
\]
We need the following lemma that can be deduced from \cite{paulin2012concentration}.
\begin{lemma}
\label{lem:VarianceMarkov}
For any $a,b,c\in\{1,\ldots,M\}$,
\begin{multline*}
\E\left[\sum_{s=1}^{p}\frac1p[\varphi_{a}(Y_{s})\varphi_{b}(Y_{s+1})\varphi_{c}(Y_{s+2})-\E\left[\varphi_{a}(Y_{1})\varphi_{b}(Y_{2})\varphi_{c}(Y_{3})\right]]\right]^{2}\\ \leq\frac{4}{p\psg}\E\left[\varphi_{a}(Y_{1})\varphi_{b}(Y_{2})\varphi_{c}(Y_{3})-\E\left[\varphi_{a}(Y_{1})\varphi_{b}(Y_{2})\varphi_{c}(Y_{3})\right]\right]^{2}\,.
\end{multline*}
\end{lemma}
\begin{proof}
Notice that $(X_{1},Y_{1}),(X_{2},Y_{2}),\ldots$ is homogenous, irreducible, aperiodic and stationary Markov chain on $\X\times\Y$, whose stationary distribution is $\tilde\pi(x,\d y)\eqdef\pi_{x}\mu_{x}(\d y)$. 
Observe that its transition kernel $\tilde\Q$ satisfies, for all $x,x'\in\X$ and all $y,y'\in\Y$,
\[
\tilde\Q(x,y;x',\d y')=\Qstar(x,x')\mu_{x'}(\d y')\,.
\]
The transition kernel $\tilde\Q$ can be viewed as an operator $\mathbb Q$ on the Hilbert space $\mathrm L^{2}(\tilde\pi)$ defined, for all $f\in\mathrm L^{2}(\tilde\pi)$, by:
\[
(\mathbb Q f)(x,y)\eqdef\E_{\tilde\Q(x,y;\ldotp,\ldotp)}(f)=\sum_{x'\in\X}\Qstar(x,x')\int_{\Y}f(x',y')\mu_{x'}(\d y')\,.
\]
Note that $\mathbb Q f(x,y)$ does not depend on $y$. Set $E\eqdef\{f(x,y)\in\mathrm L^{2}(\tilde\pi)\ :\ \text{$f$ does not depend on $y$}\}$. The $\mathrm L^{2}(\tilde\pi)$-self-adjoint operator defined, for all $f\in\mathrm L^{2}(\tilde\pi)$, by
\[
(\Pi_{E}f)(x,y)\eqdef\int_{\Y}f(x,y')\mu_{x}(\d y')\,,
\]
is the orthogonal projection onto $E$. Since $\Pi_{E}\mathbb Q\Pi_{E}=\mathbb Q$, the set of nonzero eigenvalues of $\mathbb Q$ is exactly the set of nonzero eigenvalues of the $K$ dimensional linear operator $\Pi_{E}\mathbb Q\Pi_{E}$. Eventually, note that the matrix of $\mathbb Q$ in the basis $((x,y)\mapsto\one_{x'=x})_{x'\in\X}$ is $\Qstar$. Then, the pseudo spectral gap of $\mathbb Q$ is equal to $\psg$ (the pseudo spectral gap of $\Qstar$). 

Furthermore, note the same analysis can be made for $(X_{1}, X_{2}, X_{3}, Z_{1}),(X_{2},X_{3},X_{4},Z_{2}),\ldots$ and its pseudo spectral gap is the pseudo spectral gap of the Markov chain $(X_{1}, X_{2}, X_{3}),(X_{2},X_{3},X_{4}),\ldots$ which is $\psg$. Indeed, the set of nonzero eigenvalues of the Markov chain $(X_{1}, X_{2}, X_{3}),(X_{2},X_{3},X_{4}),\ldots$ is equal to the set of nonzero eigenvalues of the Markov chain $X_{1}, X_{2},\ldots$. 

Eventually, set $g(X_{s},X_{s+1},X_{s+2},Z_{s})\eqdef(1/p)\varphi_{a}(Y_{s})\varphi_{b}(Y_{s+1})\varphi_{c}(Y_{s+2})$ and apply Theorem 3.7 in \cite{paulin2012concentration} to conclude the proof.
\end{proof}
\noindent
Then,
\begin{align*}
\E\left[\lVert\hatM-\M\lVert_{F}\right]&\leq \E\left[\lVert\hatM-\M\lVert_{F}^{2}\right]^{1/2}\,,\\
&\leq \E\left[\sum_{a,b,c}\left(\frac1p\sum_{s=1}^{p}\varphi_{a}(Y_s)\varphi_{b}(Y_{s+1})\varphi_{c}(Y_{s+2})-\E\left[\varphi_{a}(Y_{1})\varphi_{b}(Y_{2})\varphi_{c}(Y_{3})\right]\right)^{2}\right]^{1/2}\,,\\
&\leq\left[\sum_{a,b,c}\E\left(\sum_{s=1}^{p}\frac1p\left\{\varphi_{a}(Y_{s})\varphi_{b}(Y_{s+1})\varphi_{c}(Y_{s+2})-\E\left[\varphi_{a}(Y_{1})\varphi_{b}(Y_{2})\varphi_{c}(Y_{3})\right]\right\}\right)^{2}\right]^{1/2}\,,\\
&\leq\frac2{\sqrt {p\psg}}\left[\sum_{a,b,c}\E\left[\varphi_{a}(Y_{1})\varphi_{b}(Y_{2})\varphi_{c}(Y_{3})-\E\varphi_{a}(Y_{1})\varphi_{b}(Y_{2})\varphi_{c}(Y_{3})\right]^{2}\right]^{1/2}\,,\\
&\leq\left({\frac2{p\psg}}\right)^{1/2}\Big[\E\sum_{a,b,c}(\varphi_{a}(Y_{1})\varphi_{b}(Y_{2})\varphi_{c}(Y_{3})-\varphi_{a}({Y'}_{1})\varphi_{b}({Y'}_{2})\varphi_{c}({Y'}_{3}))^{2}\Big]^{1/2}\,,\\
&\leq\left({\frac{2\eta^{2}_{3}(\Phi_{M})}{p\psg}}\right)^{1/2}\,,
\end{align*}
using Jensen's inequality, Lemma \ref{lem:VarianceMarkov} and then $2\E(U-\E U)^{2}\leq\E(U-U')^{2}$ where $U$ is any real valued random variable with finite second moment and $U'$ an independent copy of $U$. The proof is similar for $\l$, $\N$ and $\Pj$.
\end{proof}

\section{Proof of Theorem \ref{thm:HSK}}
\label{sec:HSK}

\subsection*{Preliminaries lemmas}
\begin{lemma}
\label{lem:ExpressionMetP}
For all $b\in\{1,\ldots,M\}$,
\[
\M(\ldotp,b,\ldotp)=\O\D{\pistar}\Qstar\D{\O(b,\ldotp)}\Qstar\O^{\top}\,.
\]
Similarly,  $\Pj=\O\D{\pistar}{\Qstar}^{2}\O^{\top}$.
\end{lemma}
\begin{proof}
Let $a,c\in\{1,\ldots,M\}^{2}$ and observe that:
\begin{align*}
(\O\D{\pistar}\Qstar&\D{\O(b,\ldotp)}\Qstar\O^{\top})({a,c})\\
&=\sum_{(x_1,x_2,x_3)\in\X^{3}}\O(a,x_1)\pi(x_1)\Qstar(x_1,x_2)\O(b,x_2)\Qstar(x_2,x_3)\O(c,x_3)\,,\\
&=\sum_{(x_1,x_2,x_3)\in\X^{3}}\E\left[\varphi_{a}(Y_{1})|X_{1}=x_1\right]\P(X_{1}=x_1)\P(X_{2}=x_2|X_{1}=x_1)\\
&\hspace{2.5cm}\times\E\left[\varphi_{b}(Y_{2})|X_{2}=x_2\right]\P(X_{3}=x_3|X_{2}=x_2)\E\left[\varphi_{c}(Y_{3})|X_{3}=x_3\right]\,,\\
&=\E\left[\varphi_{a}(Y_{1})\varphi_{b}(Y_{2})\varphi_{c}(Y_{3})\right]\eqsp.
\end{align*}
Similarly, 
\begin{align*}
(\O\D{\pistar}{\Qstar}^{2}\O^{\top})({a,c})&=\sum_{(x_1,x_2,x_3)\in\X^{3}}\O(a,x_1)\pi(x_1)\Qstar(x_1,x_2)\Qstar(x_2,x_3)\O(c,x_3)\,,\\
&=\sum_{(x_1,x_2,x_3)\in\X^{3}}\E\left[\varphi_{a}(Y_{1})|X_{1}=x_1\right]\P(X_{1}=x_1)\P(X_{2}=x_2|X_{1}=x_1)\\
&\quad\quad\quad\quad\times\P(X_{3}=x_3|X_{2}=x_2)\E\left[\varphi_{c}(Y_{3})|X_{3}=x_3\right]\,,\\
&=\E\left[\varphi_{a}(Y_{1})\varphi_{c}(Y_{3})\right]\,,
\end{align*}
which concludes the proof.
\end{proof}

\begin{lemma}\label{lem:Fundamental}
Let $\U$ be any $(M\times K)$ matrix such that $\Pj\U$ has rank $K$. Then,
\begin{enumerate}[-]
\item 
for all $b\in\{1,\ldots,M\}$,
\[
\B(b)\eqdef(\Pj\U)^{\dag}\M(\ldotp,b,\ldotp)\U=\Rbf\D{\O(b,\ldotp)}\Rbf^{-1}\,,
\]
where $\Rbf^{-1}\eqdef\Qstar\O^{\top}\U$ and $(\Pj\U)^{\dag}\eqdef(\U^{\top}\Pj^{\top}\Pj\U)^{-1}\U^{\top}\Pj^{\top}$ denotes the Moore-Penrose pseudoinverse of the matrix $\Pj\U$ ;
\item
$\U^{\top}\Pj\U$ is invertible and, for all $b\in\{1,\ldots,M\}$,
\[
\B(b)=(\U^{\top}\Pj\U)^{-1}\U^{\top}\M(\ldotp,b,\ldotp)\U=\Rbf\D{\O(b,\ldotp)}\Rbf^{-1}\,.
\]
\end{enumerate}
\end{lemma}
\begin{proof}
Observe that $\M(\ldotp,b,\ldotp)\U=\O\D{\pistar}\Qstar\D{\O(b,\ldotp)}\Rbf^{-1}=\Pj\U\Rbf\D{\O(b,\ldotp)}\Rbf^{-1}$ as claimed.
\end{proof}

\begin{lemma}
\label{lem:Wedinsetal}
Assume that $2\lVert\hatPj-\Pj\lVert<\sigma_{K}(\Pj)$, then:
\begin{enumerate}[(i)]
\item 
\[
\varepsilon_{\Pj}\eqdef\frac{\lVert\hatPj-\Pj\lVert}{\sigma_{K}(\Pj)-\lVert\hatPj-\Pj\lVert}<1\,,
\]
\item 
\[
 \sigma_{K}(\hatPj)\geq\Big[\frac{\sigma_{K}(\Pj)-\lVert \hatPj-\Pj\lVert}{\sigma_{K}(\Pj)}\Big]\sigma_{K}(\Pj)>\frac{\sigma_{K}(\Pj)}2\,,
 \]
\item 
$
\sigma_{K}(\hat \U^{\top}\U)\geq({1-\varepsilon_{\Pj}^{2}})^{1/2}\,,
$
\item 
$
\sigma_{K}(\hat\U^{\top}\Pj\hat \U)\geq({1-\varepsilon_{\Pj}^{2}})\sigma_{K}(\Pj)\,,
$
\item for all $\alpha\in\R^{K}$ and for all $v\in\mathrm{Range}(\Pj)$, $\lVert\hat\U\alpha-v\lVert_{2}^{2}\leq\lVert\alpha-\hat\U^{\top}v\lVert_{2}^{2}+\varepsilon_{\Pj}^{2}\lVert v\lVert_{2}^{2}$,
\item if $3\lVert\hatPj-\Pj\lVert\leq\sigma_{K}(\Pj)$ then:
\[
 \sigma_{K}(\hat\U^{\top}\hatPj\hat\U) \geq\frac{\sigma_{K}(\Pj)}3\,,
 \]
 \item 
 \begin{align*}
 \lVert(\hat\U^{\top}\hatPj\hat\U)^{-1}-(\hat\U^{\top}\Pj\hat\U)^{-1}\lVert&\leq\frac{\lVert\hatPj-\Pj\lVert}{\sigma_{K}(\Pj)({1-\varepsilon_{\Pj}^{2}})(({1-\varepsilon_{\Pj}^{2}})\sigma_{K}(\Pj)-\lVert\hatPj-\Pj\lVert)}\,,\\
 &\leq3.2\frac{\lVert\hatPj-\Pj\lVert}{\sigma_{K}^{2}(\Pj)}\,.
  \end{align*}
\end{enumerate}
\end{lemma}
\begin{proof}
See Lemma C.1 in \cite{anandkumar2012method} for the first five claims. The sixth claim follows from the fourth point and Theorem \ref{thm:Weyl}. The seventh point follows from the fourth claim and Theorem \ref{thm:InversePerturbation}.
\end{proof}

\subsection*{Control of the observable operator}
Claim (iv) in Lemma \ref{lem:Wedinsetal} and Lemma \ref{lem:Fundamental} ensure that, for all $b\in\{1,\ldots,M\}$,
\[
\tilde\B(b)\eqdef(\hat\U^{\top}\Pj\hat\U)^{-1}\hat\U^{\top}\M(\ldotp,b,\ldotp)\hat\U=\tilde\Rbf\D{\O(b,\ldotp)}\tilde\Rbf^{-1}\,,
\]
where $\Rbf^{-1}$ may be defined as
\[
\tilde\Rbf^{-1}\eqdef\D{(\lVert(\Qstar\O^{\top}\hat\U)^{-1}(\ldotp,1)\lVert_{2},\ldots,\lVert(\Qstar\O^{\top}\hat\U)^{-1}(\ldotp,K)\lVert_{2})}\Qstar\O^{\top}\hat\U\,.
\]
Set $\Lambda\eqdef\Theta^{\top}\hat\U^{\top}\O$ and for all $x\in\X$, $\tilde\C(x)\eqdef\sum_{b=1}^{M}(\hat\U\Theta)(b,x)\tilde\B(b)=\tilde\Rbf\D{\Lambda(x,\ldotp)}\tilde\Rbf^{-1}$. Note that $\tilde\Rbf$ has unit Euclidean norm columns:
\[
\tilde\Rbf=(\Qstar\O^{\top}\hat\U)^{-1}\,\D{(\lVert(\Q{^{\star}}\O^{\top}\hat\U)^{-1}(\ldotp,1)\lVert_{2},\ldots,\lVert({\Qstar}\O^{\top}\hat\U)^{-1}(\ldotp,K)\lVert_{2})}^{-1}\,,
\]
corresponding to unit Euclidean norm eigenvectors of $\tilde\C(k)$.

\begin{lemma}
Assume that $3\lVert\hatPj-\Pj\lVert\leq\sigma_{K}(\Pj)$, then, for all $b\in\{1,\ldots,M\}$,
\[
\lVert\hat\B(b)-\tilde\B(b)\lVert\leq
3.2\frac{\lVert\M(\ldotp,b,\ldotp)\lVert}{\sigma_{K}(\Pj)}
\Big[\frac{\lVert\hatM(\ldotp,b,\ldotp)-\M(\ldotp,b,\ldotp)\lVert}{\lVert\M(\ldotp,b,\ldotp)\lVert}+
\frac{\lVert\hatPj-\Pj\lVert}{\sigma_{K}(\Pj)}\Big]\,,
\]
and for all $x\in\X$,
\[
\lVert{\hat\C}(x)-\tilde\C(x)\lVert\leq
3.2\frac{\lVert\M\lVert_{\infty,2}}{\sigma_{K}(\Pj)}
\Big[\frac{\lVert\hatM-\M\lVert_{\infty,2}}{\lVert\M\lVert_{\infty,2}}+
\frac{\lVert\hatPj-\Pj\lVert}{\sigma_{K}(\Pj)}\Big]\,.
\]
\end{lemma}
\begin{proof}
Observe that:
\begin{align*}
\lVert\hat\B(b)-\tilde\B(b)\lVert\leq&\lVert(\hat\U^{\top}\hatPj\hat\U)^{-1}\hat\U^{\top}\hatM(\ldotp,b,\ldotp)\hat\U-(\hat\U^{\top}\hatPj\hat\U)^{-1}\hat\U^{\top}\M(\ldotp,b,\ldotp)\hat\U\lVert\\
&+\lVert(\hat\U^{\top}\Pj\hat\U)^{-1}\hat\U^{\top}\M(\ldotp,b,\ldotp)\hat\U-(\hat\U^{\top}\hatPj\hat\U)^{-1}\hat\U^{\top}\M(\ldotp,b,\ldotp)\hat\U\lVert\,,\\
\leq&\lVert\hat\U^{\top}(\hatM(\ldotp,b,\ldotp)-\M(\ldotp,b,\ldotp))\hat\U\lVert\lVert(\hat\U^{\top}\hatPj\hat\U)^{-1}\lVert\\
&+\lVert(\hat\U^{\top}\Pj\hat\U)^{-1}-(\hat\U^{\top}\hatPj\hat\U)^{-1}\lVert\lVert\hat\U^{\top}\M(\ldotp,b,\ldotp)\hat\U\lVert\,,\\
\leq&\lVert\hatM(\ldotp,b,\ldotp)-\M(\ldotp,b,\ldotp)\lVert\sigma^{-1}_{K}(\hat\U^{\top}\hatPj\hat\U)\\
&+\lVert\M(\ldotp,b,\ldotp)\lVert\lVert(\hat\U^{\top}\Pj\hat\U)^{-1}-(\hat\U^{\top}\hatPj\hat\U)^{-1}\lVert\,.
\end{align*}
By claims (vi) and (vii) of Lemma \ref{lem:Wedinsetal}, $3\sigma_{K}(\hat\U^{\top}\hatPj\hat\U)\geq\sigma_{K}(\Pj)$ and $\lVert(\hat\U^{\top}\hatPj\hat\U)^{-1}-(\hat\U^{\top}\Pj\hat\U)^{-1}\lVert\leq3.2\frac{\lVert\hatPj-\Pj\lVert}{\sigma_{K}^{2}(\Pj)}$. Replacing $\M(\ldotp,b,\ldotp)$ by $\sum_{b=1}^{M}(\hat\U\Theta)(b,k)\M(\ldotp,b,\ldotp)$ yields the same result for $\lVert\hat\C(x)-\tilde\C(x)\lVert$.
\end{proof}
\begin{lemma}
Assume that $2\lVert\hatPj-\Pj\lVert<\sigma_{K}(\Pj)$, then,
\begin{enumerate}[(i)]
\item 
\[
\k(\tilde\Rbf)\eqdef\lVert\tilde\Rbf\lVert\lVert\tilde\Rbf^{-1}\lVert\leq\k^{2}(\Qstar\O^{\top}\hat\U)\leq\frac{\k^{2}(\Qstar\O^{\top})}{1-\varepsilon_{\Pj}^{2}}\,,
\]
\item 
\[
\sv_{\C(1)}(\hat\C(1))
\leq\k(\tilde\Rbf)\lVert\hat\C(1)-\tilde\C(1)\lVert
\leq\frac{\k^{2}(\Qstar\O^{\top})}{1-\varepsilon_{\Pj}^{2}}\lVert\hat\C(1)-\tilde\C(1)\lVert\,,
\]
where $\displaystyle\sv_{\C(1)}(\hat\C(1))\eqdef\max_{x_{1}\in\X}\min_{x_{2}\in\X}\left\lvert\hat\lambda(1,{x_{1}})-\lambda(1,{x_{2}})\right\lvert$.
\item  If in addition,
\[
\frac{\k^{2}(\Qstar\O^{\top})}{1-\varepsilon_{\Pj}^{2}}\lVert\hat\C(1)-\tilde\C(1)\lVert
<\min_{x,x'\in\X}\left\lvert\Lambda(1,x)-\Lambda(1,x')\right\lvert/2\eqsp,
\]
 then $\hat\C(1)$ has $K$ distinct real eigenvalues and:
\[
\md(\C(1),\hat\C(1))\leq\frac{\k^{2}(\Qstar\O^{\top})}{1-\varepsilon_{\Pj}^{2}}\lVert\hat\C(1)-\tilde\C(1)\lVert\,,
\]
where $\displaystyle\md(\C(1),\hat\C(1))\eqdef\min_{\tau\in\mathcal S_{K}}\left\{\max_{x\in\X}\left\lvert\hat\Lambda(1,\tau(x))-\Lambda(1,x)\right\lvert\right\}$.
\end{enumerate}
\end{lemma}
\begin{proof}
Observe that $\U$ is an orthonormal basis of range of $\O$. The first point follows from claim (iii) of Lemma \ref{lem:Wedinsetal}. The second point is derived from Theorem \ref{thm:BauerFike} and the first point. The remark following Theorem \ref{thm:BauerFike} proves the last point.
\end{proof}

\subsection*{Control of the spectra}
\begin{lemma}
For any $0<\delta<1$,
\[
\P\Big[\forall x, x_{1}\neq x_{2}\,,\ \lvert\Lambda(x,{x_{1}})-\Lambda(x,{x_{2}})\lvert\geq
{\frac
{2\delta({1-\varepsilon_{\Pj}^{2}})^{1/2}}
{\sqrt eK^{5/2}(K-1)}}
{\gamma(\O)}\Big]\geq1-\delta\,.
\]
Furthermore:
\[
\P\Big[\lVert\Lambda\lVert_{\infty}\geq\frac{1+\sqrt{2\log(K^{2}/\delta)}}{\sqrt K}\lVert\O\lVert_{2,\infty}\Big]\leq\delta\,.
\]
\end{lemma}
\begin{proof}
Observe that:
\begin{align*}
\Lambda(x,{x_{1}})-\Lambda(x,{x_{2}})&=\langle\Theta(\ldotp,x),(\hat\U^{\top}\O)(\ldotp,x_{1})-(\hat\U^{\top}\O)(\ldotp,x_{2})\rangle\\
&=\langle\Theta(\ldotp,x),\hat\U^{\top}(\O(\ldotp,x_{1})-\O(\ldotp,x_{2}))\rangle\,.
\end{align*}
Furthermore, from (iii) in Lemma \ref{lem:Wedinsetal}, we get that:
\[
\lVert\hat\U^{\top}(\O(\ldotp,x_{1})-\O(\ldotp,x_{2}))\lVert_{2}\geq({1-\varepsilon_{\Pj}^{2}})^{1/2}\lVert\O(\ldotp,x_{1})-\O(\ldotp,x_{2})\lVert_{2}\geq({1-\varepsilon_{\Pj}^{2}})^{1/2}\gamma(\O)\,.
\]
Similarly, note that:
\[
\lVert\Lambda\lVert_{\infty}=\max_{x,x'}\lvert\langle\Theta(\ldotp,x),\hat\U^{\top}\O(\ldotp,x')\rangle\lvert\,,
\]
and $\lVert\hat\U^{\top}\O(\ldotp,x')\lVert_{2}\leq\lVert\O(\ldotp,x')\lVert_{2}\leq\lVert\O\lVert_{2,\infty}$. For sake of readability, we borrow the result of Lemma F.2 and the argument of Lemma C.6 in \cite{anandkumar2012method} to conclude.
\end{proof}

\subsection*{Perturbation of simultaneously diagonalizable matrices}

\begin{lemma}
\label{LemmaHSKfin}
If $3\lVert\hatPj-\Pj\lVert\leq\sigma_{K}(\Pj)$ and:
\begin{align}
\label{eq:condition1}
8.2K^{5/2}(K-1)
\frac{\k^{2}(\Q\O^{\top})}
{\delta\gamma(\O)\sigma_{K}(\Pj)}
\Big[
\lVert\hatM-\M\lVert_{\infty,2}+
\frac{\lVert\M\lVert_{\infty,2}\lVert\hatPj-\Pj\lVert}{\sigma_{K}(\Pj)}
\Big]
&<1\,,
\\
\label{eq:condition2}
43.4K^{4}(K-1)
\frac
{\k^{4}(\Q\O^{\top})}
{\delta\gamma(\O)\sigma_{K}(\Pj)}
\Big[
\lVert\hatM-\M\lVert_{\infty,2}+
\frac{\lVert\M\lVert_{\infty,2}\lVert\hatPj-\Pj\lVert}{\sigma_{K}(\Pj)}
\Big]
&\leq1\,,
\end{align}
and for all $x, x_{1}\neq x_{2}$,
\[
\lvert\Lambda(x,{x_{1}})-\Lambda(x,{x_{2}})\lvert
\geq
{\frac
{\sqrt 3\delta}
{\sqrt eK^{5/2}(K-1)}}
{\gamma(\O)}\,,
\]
and:
\[
\lVert\Lambda\lVert_{\infty}\leq\frac{1+\sqrt{2\log(K^{2}/\delta)}}{\sqrt K}\lVert\O\lVert_{2,\infty}\,,
\]
then there exists $\tau\in\mathcal S_{K}$ such that for all $x\in\X$:
\begin{multline*}
\lVert\Lambda(\ldotp,x)-\hat \Lambda(\ldotp,\tau(x))\lVert_{\infty}\leq
\left[13\frac
{\k^{2}(\Q\O^{\top})}
{\sigma_{K}(\Pj)}+116K^{7/2}(K-1)\left\{1+\left({2\log(K^{2}/\delta)}\right)^{1/2}\right\}\right.
\\\times\left.\frac{\k^{6}(\Q\O^{\top})\lVert\O\lVert_{2,\infty}}{\delta\gamma(\O)\sigma_{K}(\Pj)}\right]\times\left[
\lVert\hatM-\M\lVert_{\infty,2}+
\frac{\lVert\M\lVert_{\infty,2}\lVert\hatPj-\Pj\lVert}{\sigma_{K}(\Pj)}
\right]\,.
\end{multline*}
\end{lemma}
\begin{proof}
Note $\varepsilon_{\Pj}\leq1/2$. Invoke the last part of Claim 4 of Lemma C.4 in \cite{anandkumar2012method} with $\gamma_{A}\leftarrow{\frac
{\sqrt 3\delta}
{\sqrt eK^{\frac52}(K-1)}}
{\gamma(\O)}$, $\kappa(R)\leftarrow\frac{4\k^{2}(\Q\O^{\top})}{3}$, $\lVert \tilde R\lVert_{2}^{2}\leftarrow\frac{4\k^{2}(\Q\O^{\top})}{3}$, $\epsilon_{A}\leftarrow3.2\frac{\lVert\M\lVert_{\infty,2}}{\sigma_{K}(\Pj)}
\Big[\frac{\lVert\hatM-\M\lVert_{\infty,2}}{\lVert\M\lVert_{\infty,2}}+
\frac{\lVert\hatPj-\Pj\lVert}{\sigma_{K}(\Pj)}\Big]$ and $\lambda_{\max}\leftarrow\frac{1+\sqrt{2\log(K^{2}/\delta)}}{\sqrt K}\lVert\O\lVert_{2,\infty}$. Observe that \eqref{eq:condition1} agrees with $\varepsilon_{3}<1/2$ and \eqref{eq:condition2} agrees with $\varepsilon_{4}\leq1/2$.
\end{proof}

\noindent
Since $\Theta^{\top}$ is an isometry, observe that:
\[
\lVert\hat\U^{\top}\O(\ldotp,x)-\Theta\hat\Lambda(\ldotp,\tau(x))\lVert_{2}=\lVert\Lambda(\ldotp,x)-\hat \Lambda(\ldotp,\tau(x))\lVert_{2}\leq\sqrt K\lVert\Lambda(\ldotp,x)-\hat \Lambda(\ldotp,\tau(x))\lVert_{\infty}\,.
\]
Claim (v) in Lemma \ref{lem:Wedinsetal} (with $\alpha=\Theta\hat\Lambda(\ldotp,\tau(x))$ and $v=\O(\ldotp,x)$) give
\begin{eqnarray*}
\lVert\O(\ldotp,x)-\hatO(\ldotp,\tau(x))\lVert_{2}
&\leq&\lVert\hat\U^{\top}\O(\ldotp,x)-\Theta\hat\Lambda(\ldotp,\tau(x))\lVert_{2}
+
\frac{3\lVert\hatPj-\Pj\lVert}{2\sigma_{K}(\Pj)}\lVert\O(\ldotp,x)\lVert_{2}
\\
&\leq&\sqrt K\lVert\Lambda(\ldotp,x)-\hat \Lambda(\ldotp,\tau(x))\lVert_{\infty}
+
\frac{3\lVert\hatPj-\Pj\lVert}{2\sigma_{K}(\Pj)}\lVert\O(\ldotp,x)\lVert_{2}.
\end{eqnarray*}
Theorem \ref{thm:HSK} follows from Lemma \ref{LemmaHSKfin}.


\begin{thebibliography}{10}

\bibitem{alexandrovich2014nonparametric}
G.~Alexandrovich and H.~Holzmann.
\newblock Nonparametric identification of hidden {M}arkov models.
\newblock {\em arXiv preprint arXiv:1404.4210}, 2014.

\bibitem{allman2009identifiability}
E.~S. Allman, C.~Matias, and J.~A. Rhodes.
\newblock Identifiability of parameters in latent structure models with many
  observed variables.
\newblock {\em Ann. Statist.}, 37(6A):3099--3132, 12 2009.

\bibitem{anandkumar2012method}
A.~Anandkumar, D.~Hsu, and S.~M. Kakade.
\newblock A method of moments for mixture models and hidden {M}arkov models.
\newblock {\em arXiv preprint arXiv:1203.0683}, 2012.

\bibitem{BMM12}
J.-P. Baudry, C.~Maugis, and B.~Michel.
\newblock Slope heuristics: overview and implementation.
\newblock {\em Stat. Comput.}, 22(2):455--470, 2012.

\bibitem{baum:petrie:soules:weiss:1970}
L.E. Baum, T.P. Petrie, G.~Soules, and N.~Weiss.
\newblock A maximization technique occurring in the statistical analysis of
  probabilistic functions of {M}arkov chains.
\newblock {\em Ann. Math. Statist.}, 41:164--171, 1970.

\bibitem{cappe:moulines:ryden:2005}
O.~Capp\'{e}, {\'E}.~Moulines, and T.~Ryd\'{e}n.
\newblock {\em Inference in Hidden {M}arkov Models}.
\newblock Springer, 2005.

\bibitem{decastro:gassiat:lacour:2015}
Y.~De~Castro, E.~Gassiat, and C.~Lacour.
\newblock Minimax adaptative estimation of non-parametric hidden {M}arkov
  models.
\newblock {\em arXiv:1501.04787 [stat.ST]}, 2015.

\bibitem{delmoral:2004}
P.~{Del Moral}.
\newblock {\em {F}eynman-Kac {F}ormulae. {G}enealogical and Interacting
  Particle Systems with Applications}.
\newblock Springer, 2004.

\bibitem{delmoral:doucet:singh:2010}
P.~Del~Moral, A.~Doucet, and S.~Singh.
\newblock A backward particle interpretation of {F}eynman-{K}ac formulae.
\newblock {\em ESAIM M2AN}, 44(5):947--975, 2010.

\bibitem{douc:garivier:moulines:olsson:2011}
R.~Douc, A.~Garivier, {\'E}.~Moulines, and J.~Olsson.
\newblock Sequential {M}onte {C}carlo smoothing for general state space hidden
  {M}arkov models.
\newblock {\em Ann. {A}ppl. {P}robab.}, 21(6):2109--2145, 2011.

\bibitem{douc:moulines:stoffer:2013}
R.~Douc, {\'E}.~Moulines, and D.~Stoffer.
\newblock {\em Nonlinear Time Series: Theory, Methods and Applications with {R}
  Examples}.
\newblock Chapman \& Hall, 2013.

\bibitem{doucet:defreitas:gordon:2001}
A.~Doucet, N.~{De Freitas}, and N.~Gordon, editors.
\newblock {\em Sequential {M}onte {C}arlo Methods in Practice}.
\newblock Springer, New York, 2001.

\bibitem{doucet:godsill:andrieu:2000}
A.~Doucet, S.~Godsill, and C.~Andrieu.
\newblock On sequential {M}onte {C}arlo sampling methods for bayesian
  filtering.
\newblock {\em Stat. Comput.}, 10:197--208, 2000.

\bibitem{dubarry:lecorff:2013}
C.~Dubarry and S.~Le~{C}orff.
\newblock Non-asymptotic deviation inequalities for smoothed additive
  functionals in nonlinear state-space models.
\newblock {\em Bernoulli}, 19(5B):2222--2249, 2013.

\bibitem{evendar:kakade:mansour:2007}
E.~Even-Dar, S.M. Kakade, and Y.~Mansour.
\newblock The value of observation for monitoring dynamic systems.
\newblock {\em IJCAI}, pages 2474--2479, 2007.

\bibitem{gassiat2013finite}
{\'E}.~Gassiat, A.~Cleynen, and S.~Robin.
\newblock Inference in finite state space non parametric hidden {M}arkov models
  and applications.
\newblock {\em Stat. Comput.}, pages 1--11, 2015.

\bibitem{godsill:doucet:west:2004}
S.~Godsill, A.~Doucet, and M.~West.
\newblock Monte {C}arlo smoothing for nonlinear time series.
\newblock {\em J. Am. Statist. Assoc.}, 50:438--449, 2004.

\bibitem{hsu2012spectral}
D.~Hsu, S.~M. Kakade, and T.~Zhang.
\newblock A spectral algorithm for learning hidden {M}arkov models.
\newblock {\em J. Comput. System Sci.}, 78(5):1460--1480, 2012.

\bibitem{hurzeler:kunsch:1998}
M.~Hurzeler and H.R. Kusch.
\newblock Monte {C}arlo approximations for general state-space models.
\newblock {\em J. Comput. Graph. Statist.}, 7:175--193, 1998.

\bibitem{kantas:doucet:singh:maciejowski:chopin:2014}
N.~Kantas, A.~Doucet, S.S. Singh, J.~Maciejowski, and N.~Chopin.
\newblock On particle methods for parameter estimation in state-space models.
\newblock {\em arXiv:1412.8695v1}, 2014.

\bibitem{kitagawa:1996}
G.~Kitagawa.
\newblock Monte-{C}arlo filter and smoother for non-{G}aussian nonlinear state
  space models.
\newblock {\em J. Comput. Graph. Statist.}, 1:1--25, 1996.

\bibitem{MR1228209}
Y.~Meyer.
\newblock {\em Wavelets and operators}, volume~37 of {\em Cambridge Studies in
  Advanced Mathematics}.
\newblock Cambridge University Press, Cambridge, 1992.
\newblock Translated from the 1990 French original by D. H. Salinger.

\bibitem{olsson:westerborn:2014}
J.~Olsson and J.~Westerborn.
\newblock Efficient particle-based online smoothing in general hidden markov
  models: the {P}a{RIS} algorithm.
\newblock {\em arXiv:1412.7550}, 2014.

\bibitem{paulin2012concentration}
D.~Paulin.
\newblock Concentration inequalities for {M}arkov chains by {M}arton couplings.
\newblock {\em arXiv preprint arXiv:1212.2015v3}, 2014.

\bibitem{rabiner:1989}
L.R. Rabiner.
\newblock A tutorial on hidden {M}arkov models and selected applications in
  speech recognition.
\newblock {\em Proc. {IEEE}}, 77:257--285, 1989.

\bibitem{stewart1990matrix}
G.~W. Stewart and J.-G. Sun.
\newblock {\em Matrix perturbation theory}.
\newblock Academic press, 1990.

\bibitem{tadic:2010}
V.B. Tadic.
\newblock Analyticity, convergence, and convergence rate of recursive
  maximum-likelihood estimation in hidden markov models.
\newblock {\em {IEEE} Transactions on Information Theory}, 56(12), 2010.

\bibitem{wedin1972perturbation}
P.-{\AA}. Wedin.
\newblock Perturbation bounds in connection with singular value decomposition.
\newblock {\em BIT Numerical Mathematics}, 12(1):99--111, 1972.

\end{thebibliography}
\end{document}